\documentclass[10pt, reqno, a4paper]{amsart}
\usepackage[latin1]{inputenc}
\usepackage{amsfonts}
\usepackage{amsmath}
\usepackage{amssymb}
\usepackage{amsthm}
\usepackage[T1]{fontenc}
\usepackage{enumerate}
\usepackage{verbatim}
\usepackage{lmodern}
\usepackage{mathtools}
\RequirePackage[dvipsnames,usenames]{color}  %makes the standard dvips color names available
\definecolor{VeryDarkGreen}{rgb}{0,0.18,0.08}
\definecolor{VeryDarkBrown}{rgb}{0.12,0.08,0.04}
\usepackage[pdftex]{hyperref}
\hypersetup{
bookmarks,
bookmarksdepth=2,
bookmarksopen,
bookmarksnumbered,
pdfstartview=FitH,
colorlinks,backref,hyperindex,
linkcolor=VeryDarkBrown,
citecolor=VeryDarkGreen,
urlcolor=VeryDarkGreen
}
\usepackage[all]{xy}
\SelectTips{cm}{10}

\newcommand{\into}{\hookrightarrow }

\DeclareMathOperator{\Supp}{Supp}

\DeclareMathOperator{\Hom}{Hom}

\DeclareMathOperator{\rk}{rk}

\DeclareMathOperator{\id}{id}
\DeclareMathOperator{\coker}{coker}
\DeclareMathOperator{\colim}{colim}

\DeclareMathOperator{\Ass}{\mathbf{Ass}}

\DeclareMathOperator{\Spec}{Spec}

\DeclareMathOperator{\im}{im}

\DeclarePairedDelimiter\abs{\lvert}{\rvert}%

\newcommand{\eps}{\varepsilon}

\title{Functorial Test Modules}
\author{Manuel Blickle}
\author{Axel St\"abler}

\begin{document}

\swapnumbers
\theoremstyle{plain}
\newtheorem{Le}{Lemma}[section]
\newtheorem{Ko}[Le]{Corollary}
\newtheorem{Theo}[Le]{Theorem}
\newtheorem*{TheoB}{Theorem}
\newtheorem{Prop}[Le]{Proposition}
\newtheorem*{PropB}{Proposition}
\newtheorem{Con}[Le]{Conjecture}
\theoremstyle{definition}
\newtheorem{Def}[Le]{Definition}
\newtheorem*{DefB}{Definition}
\newtheorem{Bem}[Le]{Remark}
\newtheorem{Bsp}[Le]{Example}
\newtheorem*{BspB}{Example}
\newtheorem{Be}[Le]{Observation}
\newtheorem{Sit}[Le]{Situation}
\newtheorem{Que}[Le]{Question}
\newtheorem{Dis}[Le]{Discussion}
\newtheorem{Prob}[Le]{Problem}
\newtheorem*{Konv}{Conventions}

\def\cocoa{{\hbox{\rm C\kern-.13em o\kern-.07em C\kern-.13em o\kern-.15em
A}}}
\address{Manuel Blickle\\
Johannes Gutenberg-Universit\"at Mainz\\ Fachbereich 08\\
Staudingerweg 9\\
55099 Mainz\\
Germany}
\email{blicklem@uni-mainz.de}

\address{Axel St\"abler\\
Johannes Gutenberg-Universit\"at Mainz\\ Fachbereich 08\\
Staudingerweg 9\\
55099 Mainz\\
Germany}
\email{staebler@uni-mainz.de}

\date{\today}


\begin{abstract}
In this article we introduce a slight modification of the definition of test modules which is an additive functor $\tau$ on the category of coherent Cartier modules. We show that in many situations this modification agrees with the usual definition of test modules. Furthermore, we show that for a smooth morphism $f \colon X \to Y$ of $F$-finite schemes one has a natural isomorphism $f^! \circ \tau \cong \tau \circ f^!$. If $f$ is quasi-finite and of finite type we construct a natural transformation $\tau \circ f_* \to f_* \circ \tau$.
\end{abstract}

\maketitle

\section*{Introduction}
Since their appearance in tight closure theory in the nineties \cite{hochsterhunekebriancon}, test ideals played an ever more important role in the study of singularities of algebraic varieties over a field of positive characteristic. Their connection to multiplier ideals, first explored by Smith \cite{smithtestmultiplier} and Hara \cite{haratestmultiplier}, made them a key ingredient in the dictionary between the singularity types in characteristic zero arising from the minimal model program and the so-called $F$-singularities in positive characteristic. There have been various generalizations of the original definition, either incorporating additional data (test ideals for pairs \cite{haratakagigeneralization} and triples \cite{blickleschwedetakagizhangDiscretenessfjumpingsingular}) or allowing ever more general ground rings \cite{schwedenonqgorenstein}, or working with the canonical sheaf instead of the ring itself \cite{smithrational}. In \cite{blicklep-etestideale} the first author, building on work of Schwede \cite{schwedenonqgorenstein}, gave a very general framework in the context of Cartier Modules \cite{blickleboecklecartierfiniteness} which allowed the definition of so-called test modules, generalizing most approaches of test ideals considered previously. Besides many conceptual advantages, one shortcoming of this definition was that it is neither additive (does not preserve direct sums) nor functorial. In this article we study a slight variation of the original definition in \cite{blicklep-etestideale} which was suggested by Karl Schwede, which remedies these two shortcomings and at the same time agrees with the original definition in many key situations.

Let us briefly outline the original definition of test modules to pinpoint the necessary change in order to achieve this goal. For simplicity let us assume that $R = \mathbb{F}_p[x_1, \ldots, x_n]$ and $f$ is the equation of a hypersurface. We denote by $F_\ast R$ the ring $R$ considered as a module over itself via the Frobenius morphism $F \colon r \mapsto r^p$. In this case a basis of $F_\ast R$ over $R$ is given by $x_1^{i_1} \cdots x_n^{i_n}$ where each $i_j \leq p-1$. We denote by $\kappa: F_\ast R \to R$ the Cartier operator sending the basis monomial $x_1^{p-1} \cdots x_n^{p-1}$ to $1$ and all others to zero. Then the test ideal $\tau(R, f^t)$, for $t \in \mathbb{R}_{\geq 0}$, is defined as the smallest non-zero ideal that is stable under the maps $\kappa^e f^{\lceil tp^e \rceil}$ for all $e \geq 0$. It is known that varying $t$ one obtains a descending filtration of ideals of $R$. The numbers $t \in \mathbb{R}$ where $\tau(R, f^t) \neq \tau(R, f^{t-\eps})$ for all $\eps > 0$ are called $F$-jumping numbers of the hypersurface $f$. By \cite{blicklemustatasmithdiscretenesshypersurfaces} these form a discrete set of rational numbers.

To generalize this definition to modules one replaces the ring $R$ by a finitely generated $R$-module $M$ and the set of maps $\kappa^e f^{\lceil t p^e \rceil}$ by a graded subalgebra $\mathcal{C}$ of $\bigoplus_{e \geq 0} \Hom_R(F^e_\ast M, M)$ such that $\mathcal{C}_0 = R$. Denoting the positively graded part of $\mathcal{C}$ by $\mathcal{C}_+=\bigoplus_{e \geq 1} \Hom_R(F^e_\ast M, M)$ one obtains a descending sequence
\[
    M \supseteq \mathcal{C}_+M \supseteq \mathcal{C}^2_+ M \supseteq \mathcal{C}^3_+M \supseteq \ldots
\]
of $\mathcal{C}$-submodules which by \cite{blicklep-etestideale} stabilizes. We denote the stable member by $\underline{M}$. Then the test module $\tau(M, \mathcal{C})$ as defined in op.\ cit.\ is the smallest $\mathcal{C}$-submodule $N \subseteq \underline{M}$ such that for every generic point $\eta$ of $\Supp \underline{M}$ one has an equality $N_\eta = \underline{M}_\eta$. Replacing $M$ by $\underline{M}$ is a crucial point which already hints at the fact that a desired theory of test modules does not discriminate between Cartier modules $N \to M$ where on the kernel and cokernel some high enough power $\mathcal{C}^n_+$ acts as zero, i.e. $N \to M$ is a nil-isomorphism. One advantage of the generality obtained in this way is that one can reduce questions of test ideals on singular rings $R/I$ to test modules over the better behaved (e.g.\ regular) ambient ring $R$, and vice versa. In particular this allows to define test ideals for non-reduced rings \cite{blicklep-etestideale}. However, the following example brings to light some shortcomings:

\begin{BspB}
Consider $R = \mathbb{F}_p[x,y]$ and the inclusion $M = R/(y) \to R/(y) \oplus R = N$. We endow $M = \mathbb{F}_p[x]$ with the Cartier operator as defined above and likewise on the first summand of $N$. On the other summand the operator $\kappa$ acts by sending $y^{p-1}$ to $1$ and all other monomials of the basis to $0$. This defines the structure of a Cartier module on $M$ and $N$ and the natural inclusion into the first factor $N \subseteq M$ is compatible with the action of $\kappa$ just defined. One easily checks that $M$ admits no proper non-zero submodules stable under $\kappa$. Hence, $\tau(M, 1) = M$. On the other hand, the generic point of $N=\underline{N}$ is $(0)$. But clearly $0 \oplus (x)$ is a submodule stable under the just defined operation of $\kappa$. Hence, $\tau(N, 1) \subseteq (0) \oplus (x)$, and in fact, equality holds. Therefore we see that the inclusion $M \subseteq N$ does not induce an inclusion of test modules, nor is the test module of the direct sum $N$ the sum of test modules of its summands.
\end{BspB}

The source of these troubles is the fact that the definition of test modules does not account for non-minimal associated primes. By modifying the definition appropriately to also incorporate non-minimal associated primes this can be avoided, following a suggestion of Schwede:

\begin{DefB}
Let $R$ be a noetherian $F$-finite ring, $\mathcal{C}$ an $R$-Cartier algebra and $M$ a $\mathcal{C}$-module which is coherent as an $R$-module. The \emph{test module} $\tau(M, \mathcal{C})$ is the smallest $\mathcal{C}$-submodule $N$ of $M$ such that for every associated prime $\eta$ of $M$ the inclusion $H^0_\eta(N_\eta) \subseteq H^0_\eta(M_\eta)$ is a nil-isomorphism.
\end{DefB}

The preceding paragraph roughly explains the notation of Cartier algebra and nil-isomorphism, for more details see \autoref{SectionDefBasic}. $H^0_\eta$ denotes the $\eta$-power torsion functor. Since, in particular, $\underline{M} \subseteq M$ is a nil-isomorphism one may replace $M$ by $\underline{M}$ in the above definition. The following summarizes our key foundational results.

\begin{TheoB}
Let $R$ be a noetherian $F$-finite ring and $\mathcal{C}$ an $R$-Cartier algebra.
\begin{enumerate}
\item Theorem \autoref{TauImpliesTestElements} and Theorem \autoref{TestElementExistence}: Under mild assumptions (e.g. $R$ is of finite type over a field) the test module $\tau(M,\mathcal{C}) \subseteq M$ exists for all coherent $\mathcal{C}$-modules $M$, and there is a theory of test elements.
\item Theorem \autoref{TestModulesAreFunctorial}: The inclusion $\tau(M,\mathcal{C}) \subseteq M$ is a natural transformation of additive functors $\tau \to \id$ on the category of coherent $\mathcal{C}$-modules.
\item Theorem \autoref{OldNewTestModuleRelation}: Temporarily denoting by $\tau'(M,\mathcal{C})$ the test module as defined in \cite{blicklep-etestideale}, one has an inclusion $\tau' \subseteq \tau$ which is an equality if all associated primes of $\underline{M}$ are minimal. In particular, for $M=R$ (the test-ideal case) the two definitions coincide.
\end{enumerate}
\end{TheoB}
Furthermore, also in the case where $(M, \kappa)$ is a so-called $F$-regular Cartier module (in the sense of \cite[Definition 3.4]{blicklep-etestideale}) and $f$ a non zero-divisor both versions of the test module $\tau'(M,\mathcal{C}) = \tau(M,\mathcal{C})$ agree for $\mathcal{C}$ the Cartier algebra generated by the $\kappa f^{\lceil tp^e\rceil}$ for $e \geq 1$. For this class of Cartier modules the authors constructed in \cite{blicklestaeblerbernsteinsatocartier} Bernstein-Sato polynomials and the second author also imposed similar conditions of $F$-regularity on many results of his construction of a $V$-filtration for Cartier modules in \cite{staeblertestmodulnvilftrierung}. One objective of our new definition of test modules here is to be able remove the $F$-regularity assumptions in these instances. But this will not be discussed in this paper.

Together with setting up the notation of Cartier modules and developing a theory of ``associated primes up to nilpotence'', the results mentioned so far are contained in the first four sections of this paper. Also included is a theory of test elements which is a crucial technical tool for many computations of test modules, and we will make ample use of them in Sections \ref{SectionStructuralResults}, \ref{SectionTestmodulesPullbacks} and \ref{SectionTestmodulesPushforwards}. We conclude this part with some structural results in Section \ref{SectionStructuralResults} such as discreteness and a Skoda-type theorem and some related notions of test modules analogous to the classical case.

In Section \ref{SectionCartieralgebraPullbackPushforward} we construct the (twisted) inverse and direct image for Cartier modules, generalizing the case of a principal Cartier algebra treated in \cite{blickleboecklecartierfiniteness}. Given a morphism $f: X \to Y$ and a Cartier algebra $\mathcal{C}_Y$ on $Y$ we define the structure of a Cartier algebra on $X$ on the pullback $f^*\mathcal{C}_Y$. This allows us to show that the twisted inverse image functor $f^!$ on $\mathcal{O}$-modules naturally restricts to a functor on $\mathcal{C}$-modules in the case where $f$ is either finite or essentially smooth, i.e. if $M$ is a $\mathcal{C}_Y$-module, then $f^!M$ naturally carries the structure of a $f^*\mathcal{C}_Y$-module. Furthermore, if $\mathcal{C}_X=f^*\mathcal{C}_Y$ the same is the case for the pushforward $f_*$, i.e. if $N$ is a $\mathcal{C}_X$-module then $f_*N$ naturally carries a structure of a $\mathcal{C}_Y$-module. The following theorem summarizes the behavior of the test module functor $\tau$ with respect to pullback and pushforward. The proofs of these results occupy Sections 6--8.

\begin{TheoB}
Let $f \colon X \to Y$ a morphism of $F$-finite schemes. For $\mathcal{C}_Y$ a Cartier algebra on $Y$ set $\mathcal{C}_X =f^*\mathcal{C}_Y$.
\begin{enumerate}
\item The pushforward $f_*$ on quasi-coherent sheaves induces a functor from $\mathcal{C}_X$-modules to $\mathcal{C}_Y$-modules which preserves nilpotence. If $f$ is of finite type, and under some mild boundedness assumption for $\mathcal{C}_Y$, the functor $f_*$ preserves coherence up to nilpotence.
\item (Theorem \ref{UpperShriekCartierStructure}) In each of the three following cases
 \begin{enumerate}
 \item$f$ is essentially \'etale and $f^!=f^*$
\item $f$ finite and $f^!=\bar{f}^{\ast}\mathcal{H}om_Y(f_*\mathcal{O}_X,\phantom{M})$
\item $f$ smooth and $f^! = f^* \otimes \omega_{X/Y}$
\end{enumerate}
the pullback $f^!$ on quasi-coherent sheaves  induces a functor $\mathcal{C}_Y$-modules to $\mathcal{C}_Y$-modules which preserves nilpotence and coherence.
\item (Corollary \ref{TauCommutesSmoothPullback}) If $f$ is smooth, then there is a natural isomorphism of functors $f^! \circ \tau \cong \tau \circ f^!$.
\item (Proposition \ref{FiniteShriekTauInclusion}) If $f$ is finite and dominant then one has a natural inclusion $\tau \circ f^! \into f^! \circ \tau$.
\item (Propositions \ref{FiniteMorphismTauNatTransf} and \ref{TauFiniteDomPushforward}) If $f$ is finite then one has a natural isomorphism $f_* \circ \tau \to \tau \circ f_*$.
\item (Proposition \ref{OpenImmersionTauNatTransf} and Theorem \ref{TauQuasifinitepushforward}) If $f$ is an open immersion or $f$ is quasi-finite and of finite type (and under some mild technical assumption on $\mathcal{C}_Y$) on has a natural inclusion $\tau \circ f_* \into f_* \circ \tau$.
\end{enumerate}
\end{TheoB}

There is, by now, ample evidence that the associated graded of the test module filtration is connected to the \'etale $p$-torsion nearby cycles functor (cf.\ \cite{staeblertestmodulnvilftrierung}, \cite{stadnikvfiltrationfcrystal}, \cite{staeblerunitftestmoduln}, \cite{blicklestaeblerbernsteinsatocartier}). However, it also seems to avoid pathologies that only occur for $\ell$-torsion nearby cycles with $\ell = p$. The connection is similar to that of the multiplier ideal filtration with nearby cycles in the complex case. Namely, the multiplier ideal filtration is a sub-filtration of the $\mathcal{D}$-module theoretic $V$-filtration. The associated graded of the latter in the range $[0,1)$ in turn corresponds, via the Riemann-Hilbert correspondence, to perverse complex nearby cycles.

In characteristic $p > 0$ the replacement for the Riemann-Hilbert correspondence comes in two steps: First there is an equivalence of (abelian) categories (see \cite{blickleboecklecartierfiniteness}, \cite{emertonkisinrhunitfcrys}): \[\{ \kappa\text{-crystals}\} \to \{\text {locally finitely generated unit}(=\text{lfgu}) \, R[F]\text{-modules}\},\] where unit $R[F]$-modules are certain well-behaved $\mathcal{D}$-modules (see \cite{emertonkisinintrorhunitfcrys}). This is followed by the correspondence (on the level of appropriate derived categories) of Emerton and Kisin (\cite{emertonkisinrhunitfcrys})
\[\{ \text{lfgu } R[F]\text{-modules}\} \to \{\text{perverse constructible } p\text{-torsion sheaves}\}\]
which is an \emph{anti}-equivalence. Under this Riemann-Hilbert-type correspondence the functors $f^!$ and $f_\ast$ on $\kappa$-crystals correspond to $f^\ast$ and $f_!$. The category of $\kappa$-crystals is mapped to perverse constructible sheaves (in the sense of \cite{gabbertstructures}) by this equivalence.

In the special case that we start with a Cartier module of the form $(M, \kappa)$ we may associate a principal Cartier structure to $Gr^t = \tau(M, f^{t - \eps})/\tau(M, f^t)$ (cf. Remark \ref{AssGradedRemark} below, or \cite[Sections 4 and 5]{staeblertestmodulnvilftrierung}). Then as corollaries (\ref{GrCommutesSmoothPullback}, \ref{FiniteShriekGrInclusion}, \ref{GrQuasifinitepushforward}) to the above theorem we obtain a natural transformation of $\kappa$-crystals $Gr^t \circ f^! \to f^! \circ Gr^t$ for $f$ finite or smooth which in the latter case is an isomorphism and a natural transformation of $\kappa$-crystals $Gr^t \circ f_\ast \to f_\ast \circ Gr^t$ for $f$ quasi-finite of finite type which is an isomorphism if $f$ is proper. Via the Riemann-Hilbert correspondence we get corresponding natural transformations of perverse constructible sheaves by reversing the arrows.

Analogous transformations are true in the $\ell \neq p$ context: For $\ell$-adic nearby cycles $R \psi$  one has a natural transformation $f^\ast R \psi \to R \psi f^\ast$ which is an isomorphism if $f$ is smooth and for $f$ quasi-finite a natural transformation $f_! R \psi \to R \psi f_! $ which is an isomorphism if $f$ is proper (see \cite[Expos\'e XIII, 2.1]{SGA7II}). This is further evidence supporting the expectation that the perverse constructible sheaf corresponding to $Gr^{[0,1)}=\bigoplus_{t\in[0,1)} Gr^t$ is a suitable, possibly better behaved, replacement for the problematic $p$-torsion nearby cycles functor.

\subsection*{Conventions} We will work with noetherian rings containing a field of positive characteristic. Throughout the letter $F$ denotes the Frobenius morphism. We call a scheme $X$ \emph{$F$-finite} if the Frobenius morphism $F: X \to X$ is a finite morphism. In other words, $F_\ast \mathcal{O}_X$ is a finite $\mathcal{O}_X$-module.

\subsection*{Acknowledgements}
Both authors were supported by SFB/Transregio 45 Bonn-Essen-Mainz financed by Deutsche Forschungsgemeinschaft.

\section{Definition and basic properties}
\label{SectionDefBasic}
In this section we introduce a new and better behaved definition of the test module which, contrary to the original definition (cf. \cite[Definition 3.1, Remark 3.3]{blicklep-etestideale}), is an additive functor on the category of Cartier modules. Then we redevelop the basic theory of test modules in analogy with \cite{blicklep-etestideale}.

We start by recalling the definition of Cartier modules and some of their basic properties.

\begin{Def}
\label{CartierAlgebraDefRing}
Let $R$ be a commutative ring containing a field of characteristic $p > 0$. A \emph{Cartier algebra $\mathcal{C}$ over a ring $R$} (or \emph{$R$-Cartier algebra}) is an $\mathbb{N}$-graded ring $\bigoplus_{e \geq 0} \mathcal{C}_e$ with an $R$-bimodule structure which for a homogeneous element $\kappa \in \mathcal{C}_e$ and $r \in R$ satisfies $r \kappa = \kappa r^{p^e}$. Moreover, we require that $\mathcal{C}_0 = R$.

We say that $\mathcal{C}_R$ is \emph{finitely generated} if there are elements $\kappa_1, \ldots, \kappa_n \in \mathcal{C}_R$ such that the monomials in the $\kappa_i$ form a set of right $R$-module generators of $\mathcal{C}_R$.
\end{Def}

Note that this terminology is slightly unfortunate, since $\mathcal{C}$ is \emph{not} an $R$-algebra in the usual sense. Indeed, $R$ is in general not in the center of $\mathcal{C}_R$ if $R \neq \mathbb{F}_p$.

By convention, a $\mathcal{C}$-module is always a left $\mathcal{C}$-module, whose underlying $R$-module is finitely generated, unless stated otherwise. Synonymously to $\mathcal{C}$-module we will frequently use \emph{Cartier module} (if $\mathcal{C}$ is clear form the context) or \emph{$R$-Cartier-module} (if the ring $R$ needs to be emphasized).

We set $\mathcal{C}_+ = \bigoplus_{e \geq 1} \mathcal{C}_e$ and denote the $e$-times iterated product $(\mathcal{C}_+)^e$ by $\mathcal{C}_+^e$.
A Cartier module $M$ is called \emph{nilpotent} if $\mathcal{C}_+^e M = 0$ for some (equivalently all) $e \gg 0$. One can show (cf. \cite[Section 2.1]{blicklep-etestideale} for the case of an algebra and \cite{blickleboecklecartierfiniteness} for a more detailed treatment in the case of a single morphism) that nilpotent Cartier modules form a Serre-subcategory of the category of coherent Cartier modules. A morphism $\varphi \colon M \to N$ of $\mathcal{C}$-modules is called a \emph{nil-isomorphism} if both $\ker \varphi$ and $\coker \varphi$ are nilpotent. The category of $\emph{Cartier crystals}$ is obtained by localizing the category of Cartier modules at the Serre-subcategory of nilpotent Cartier modules; i.e.~it is obtained by formally inverting nil-isomorphisms. This category of \emph{Cartier crystals} is again an abelian category. The objects are the same as those of the category of Cartier modules. However, a morphism $M \to N$ is given by a diagram $M \xleftarrow{\varphi} M' \to N$ in Cartier modules, where $\varphi$ is a nil-isomorphism. We refer the reader to \cite{staeblertestmodulnvilftrierung} for more background on the connection of Cartier crystals and test modules and to \cite{blickleboecklecartierfiniteness} for more background on Cartier crystals.

A crucial result in the theory of coherent Cartier modules is the fact that the descending chain $\mathcal{C}_+ M \supseteq \mathcal{C}_+^2 M \supseteq \cdots$ stabilizes (cf.\ \cite[Proposition 2.13]{blicklep-etestideale}). From this fact one obtains the following

\begin{Prop}
Let $M$ be a coherent $\mathcal{C}$-module. Then there exists a unique $\mathcal{C}$-submodule $\underline{M}$ such that
\begin{enumerate}[(a)]
 \item The quotient $M/\underline{M}$ is nilpotent (in particular, the crystals associated to $M$ and $\underline{M}$ are isomorphic).
\item $\mathcal{C}_+ \underline{M} = \underline{M}$ (i.e.\ $\underline{M}$ does not admit nilpotent quotients).
\end{enumerate}
\end{Prop}
\begin{proof}
One verifies that the stable member $\mathcal{C}_+^e M$ for $e \gg 0$ satisfies these conditions.
\end{proof}

\begin{Def}
We call a coherent Cartier module $M$ $F$-pure if $\underline{M} = M$. This is equivalent to the condition that $\mathcal{C}_+ M = M$.
\end{Def}

Before we turn to a study of associated primes of a Cartier module we revisit the example from the introduction. We will illustrate that the lack of functoriality can also not be remedied by passing to the crystal of the associated graded of the test module filtration.

\begin{Bsp}
\label{Bsp1}
Consider the inclusion $k[x,y]/(x) \to k[x,y]/(x) \oplus k[x,y]$, where the Cartier structure is given by $\kappa x^{p-1}$ acting diagonally. If we consider the test module filtration of \cite{blicklep-etestideale} along $y$ then $\tau(k[y], y^t) = (y^{\lfloor t \rfloor})$ while $\tau(k[y] \oplus k[x,y], y^0)$ is given by $(0) \oplus (x)$. In particular, $Gr^1(k[y]) := \tau(k[y], y^{1 - \eps})/\tau(k[y], y^1) = k$ as a crystal. Since $\tau$ is a decreasing filtration $Gr^1$ of the direct sum has to be zero in the first component. Thus we do not obtain an induced map from the inclusion on the level of Cartier crystals either.
\end{Bsp}

The issue here is of course -- as mentioned in the introduction -- that the  definition in op.~cit.~only considers generic points and ignores non-minimal associated primes.

Note that given a Cartier module $M$ and an ideal $I$ in $R$ the $I$-power-torsion $H^0_I(M)$ of $M$ is again a Cartier module. This is simply due to the observation that for $\kappa \in \mathcal{C}^e$ and $m \in M$ with $Im=0$ one has $I \cdot \kappa(m)= \kappa(I^{[p]}m)=\kappa(0)=0$. This observation leads to the following definition.

\begin{Def}
For a Cartier module $M$ we denote by $\Ass M$ the set of primes $\eta$ of $\Spec R$ for which $H^0_\eta(M_\eta)$ is not nilpotent. We refer to $\Ass M$ as the set of \emph{associated primes} of $M$.
\end{Def}

In a sense to be made precise below, $\Ass M$ may be viewed as the associated primes of the Cartier Crystal underlying $M$. It is a subset of the usual associated primes of the underlying $R$-module $M$.

\begin{Le}
\label{AssForCrystals}
Let $0 \to P \to M \to N \to 0$ be a short exact sequence of Cartier modules.
\begin{enumerate}[(a)]
\item{Then $\Ass P \subseteq \Ass M \subseteq \Ass P \cup \Ass N$.}
\item{If the sequence splits, i.e.\ $M = N \oplus P$, then $\Ass M = \Ass N \cup \Ass P$.}
\item{If $S$ is a multiplicative set then $\Ass S^{-1}M = \{\eta \in \Ass M \, \vert \, \eta \cap S = \varnothing\}$.}
\end{enumerate}
\end{Le}
\begin{proof}
The first inclusion of part (a) is clear. For the other inclusion assume that $\eta \in \Ass M \setminus \Ass P$. Localizing at $\eta$ and applying $H^0_\eta$ to the short exact sequence we get $0 = H^0_\eta(P_\eta) \to H^0_\eta(M_\eta) \to H^0_\eta(N_\eta)$ which shows that $H^0_\eta(N_\eta)$ is not nilpotent, since its submodule $H^0_{\eta}(M_\eta)$ is not nilpotent. Part (b) follows from (a) since the sequence splits in this case.

For part (c) note that the statement is true for the associated primes of the underlying modules (\cite[Theorem 3.1 (c)]{eisenbud}). Hence, given a prime $\eta$ with $\eta \cap S = \varnothing$ we have to show that $H^0_\eta(M_\eta)$ is nilpotent if and only if $H^0_\eta(S^{-1}M_\eta)$ is nilpotent. But clearly $M_\eta = S^{-1}M_\eta$ since $S \subseteq R \setminus \eta$.
\end{proof}

\begin{Prop}
\label{Nilisocriterion}
A morphism $\varphi: M \to N$ of Cartier modules is a nil-isomorphism if and only if there is $e$ such that for every homogeneous element $\kappa \in \mathcal{C}_+^e$ of degree $d$ there exists an $R$-linear map $\alpha: F_\ast^{2d} N \to M$ making the following diagram commutative.
\[
\begin{xy} \xymatrix{ F_\ast^{2d} M \ar[d]^{\kappa^2} \ar[r]^{F_\ast^{2d} \varphi} &F_\ast^{2d} N \ar[dl]^\alpha \ar[d]^{\kappa^2} \\ M \ar[r]^{\varphi} & N}\end{xy}
\]

\end{Prop}
\begin{proof}
The map $\varphi$ is a nil-isomorphism if and only if $\ker \varphi$ and $\coker \varphi$ are nilpotent. This in turn means that we find $e \geq 1$ such that $\mathcal{C}_+^e \ker \varphi = 0$ and $\mathcal{C}_+^e \coker \varphi = 0$. This is the case if and only if every \emph{homogeneous} element of $\mathcal{C}_+^e$ operates trivially on both $\ker \varphi$ and $\coker \varphi$. In this case, if $\kappa \in \mathcal{C}_+^e$ is homogeneous, then $\varphi$ is a nil-isomorphism for the Cartier algebra generated by $\kappa$. Now we deduce from \cite[Proposition 2.3.2]{blickleboecklecartiercrystals} (cf.\ \cite[Proposition 3.3.9]{boecklepinktaucrystals} for a proof and the explicit bound) that this is the case if and only if there is $\alpha$ as desired.
\end{proof}

\begin{Le}
\label{H0FunctorOnCrys}
If $\varphi: M \to N$ is a nil-isomorphism of Cartier modules then $H^0_I(\varphi)$ is a nil-isomorphism for any ideal $I$. In particular, $H^0_I$ induces a functor on crystals.
\end{Le}
\begin{proof}
Since $H^0_I$ is a functor on Cartier modules it preserves nil-isomorphisms by Proposition \ref{Nilisocriterion}. The second claim is then due to \cite[Proposition 2.2.4]{boecklepinktaucrystals}.
\end{proof}

\begin{Le}
\label{AssDescendsToCrystals}
If $\varphi: M \to N$ is a nil-isomorphism of Cartier modules then $\Ass M = \Ass N$.
\end{Le}
\begin{proof}
Using Lemma \ref{H0FunctorOnCrys} above and localizing we obtain that $H^0_\eta(M)_\eta$ is nilpotent if and only if $H^0_\eta(N)_\eta$ is nilpotent.
\end{proof}

\begin{Def}
\label{DefTestmodule}
For a  $\mathcal{C}$-module $M$ we define the \emph{test module} $\tau(M, \mathcal{C}) = \tau(M)$ as the smallest $\mathcal{C}$-submodule $N$ of $M$ such that for every associated prime $\eta$ of $M$ the inclusion $H^0_\eta(N_\eta) \subseteq H^0_\eta(M_\eta)$ is a nil-isomorphism.
\end{Def}
Since the inclusion $\underline{M} \subseteq M$ is a nil-isomorphism, one could replace $M$ by $\underline{M}$ in the above definition. Since for primes $\eta$ not associated to $M$ one has $H^0_\eta(M_{\eta})=0$, one also might range over all primes of $R$ in the above definition.
\begin{Bem}
One could also simply work with the associated primes of the underlying module and require an equality $H^0_\eta(N_\eta) = H^0_\eta(\underline{M}_\eta)$ instead of merely a nil-isomorphism. However, contrary to the theory of test modules as developed in \cite{blicklep-etestideale} this notion does not pass to crystals (which in the end is essential when one wants to transport these constructions to unit $R[F]$-modules and constructible $\mathbb{F}_p$-sheaves). The key fact why the construction of \cite{blicklep-etestideale} is independent of the particular choice of Cartier module is that if $M$ is $F$-pure then the support of the associated crystal is the same as the underlying support of the module (cf.\ \cite[Theorem 1.8]{staeblerunitftestmoduln} for details).

The same is not true for associated primes however. Namely, consider the Cartier module $M = \omega_R \oplus \omega_R/(x)\omega_R$ over the ring $R = \mathbb{F}_p[x]$, $p \geq 3$, with Cartier structure given by $(m, n) \mapsto (\kappa(xm), \kappa(x^{p-1} m) + (x) \omega_R)$. Then $M$ is $F$-pure since if $\delta$ is a generator of $\omega_R$ then $(x^{p-2} \delta, 0)$ is mapped to $(\delta,0)$ and $( \delta, 0)$ is mapped to $(0, \delta + (x)\omega_R)$. Furthermore, $M$ admits the nilpotent submodule $N = \{(0, m) \mod (x) \omega_R \, \vert \, m \in \omega_R\}$ which corresponds to the associated non-minimal prime $(x)$. Hence, for the $F$-pure Cartier module $M$ we have $\Ass M = \{(0)\}$ while the associated primes of the underlying module are $\{(0), (x)\}$.
\end{Bem}

A slight variation of this example shows that $H^0_\eta(M)$ need not be $F$-pure if $M$ is $F$-pure and $\eta \in \Ass M$:
\begin{Bsp}
\label{FpurityNotPreservedByH^0}
Take $R = \mathbb{F}_p[x,y], p \geq 3$ and $M = \omega_R \oplus \omega_R/(x)\omega_R$ with Cartier structure $(m,n) \mapsto (\kappa(xm), \kappa(x^{p-1}m) + y\kappa(n))$. Since $(x^{p-2} y^{p-1} \delta, 0)$ is mapped to $(\delta, 0)$ and $(y^{p-1} \delta, 0)$ is mapped to $(0, \delta)$ we conclude that $M$ is $F$-pure. But the submodule $H^0_{(x)}(M) = \omega_R/(x)\omega_R$ is not $F$-pure.
\end{Bsp}

\begin{Def}
\label{DefFRegular}
We say that a $\mathcal{C}$-module $M$ is \emph{$F$-regular} if it is $F$-pure (i.e.\, $\mathcal{C}_+ M = M$) and if it contains no proper submodule $N$ for which the inclusion $H^0_\eta(N_\eta) \subseteq H^0_\eta(M_\eta)$ is a nil-isomorphism for all associated primes $\eta$ of $M$.
\end{Def}

In other words, $M$ is $F$-regular if and only if $M = \tau(M)$.

\begin{Le}
\label{ImageofFregularisFRegular}
Let $\varphi: M \to N$ be a surjective homomorphism of Cartier modules and assume that $M$ is $F$-regular. Then $N$ is $F$-regular.
\end{Le}
\begin{proof}
Note that $N$ is automatically $F$-pure since $\varphi$ is surjective.
Let $N'$ be a Cartier submodule of $N$ such that $H^0_\eta(N'_\eta) \subseteq H^0_\eta(N_\eta)$ is a nil-isomorphism for all associated primes $\eta$ of $N$. Denoting the pre-image of $N'$ by $M'$ we obtain the following sequence whose rows are exact.
\[\begin{xy}
  \xymatrix{0 \ar[r]& \ker \varphi \ar[r] & M \ar[r]^\varphi& N \ar[r]&0\\
0 \ar[r]& \ker \varphi \ar[u]^{\id} \ar[r] & M' \ar[u] \ar[r]^\varphi& N' \ar[u] \ar[r]&0}
  \end{xy}
\]
Let $\nu$ be in $\Ass M$. If $\nu$ is not contained in $\Ass N$ then applying $H^0_\nu$ and localization to the above sequences we obtain that $H^0_\nu(N'_\nu)$ and $H^0_\nu(N_\nu)$ are both nilpotent since $\Ass N' \subseteq \Ass N$. Consequently, the inclusions $H^0_\nu(\ker \varphi_\nu) \subseteq H^0_\nu(M_\nu)$ and $H^0_\nu(\ker \varphi_\nu) \subseteq H^0_\nu(M'_\nu)$ are nil-isomorphisms. We conclude that the inclusion $H^0_\nu(M'_\nu) \subseteq H^0_\nu(M_\nu)$ is a nil-isomorphism.

If $\nu$ is contained in $\Ass N$ then $H^0_\nu(N'_\nu) \subseteq H^0_\nu(N_\nu)$ is a nil-isomorphism. The long exact sequence of local cohomology (we only need the first $H^1$-term) and the $5$-lemma yield the nil-isomorphism $H^0_\nu(M'_\nu) \subseteq H^0_\nu(M_\nu)$.\footnote{Just as the $I$-power-torsion functor $H^0_\eta$ is a functor on Cartier modules, its higher derived functors, the local cohomology functors $H^i_\eta(M)$, are also functors on Cartier modules. This follows directly from the observation that the \v{C}ech complex used to compute local cohomology is in fact a complex of Cartier modules. See \cite{blickleboecklecartierfiniteness} for details.} Altogether we conclude that $H^0_\nu(M'_\nu)$ is nil-isomorphic to $H^0_\nu(M_\nu)$ for all $\nu \in \Ass M$ so that by $F$-regularity $M = M'$.
\end{proof}

\begin{Prop}
\label{TestModulesAreFunctorial}
Let $\varphi: M \to N$ be a homomorphism of Cartier modules and assume that $\tau(M)$ and $\tau(N)$ exist. The following hold:
\begin{enumerate}[(a)]
\item{Taking the test module is functorial and $\tau(\varphi) := \varphi\vert_{\tau(M)}$ induces the map $\tau(M) \to \tau(N)$.}
\item{Taking the test module commutes with finite direct sums, i.e.\ $\tau(M_1 \oplus \ldots \oplus M_n) = \tau(M_1) \oplus \ldots \oplus \tau(M_n)$.}
\end{enumerate}
In other words, the inclusion $\tau(M) \subseteq M$ is an additive sub-functor $\tau \to \id$.
 \end{Prop}
\begin{proof}
 \begin{enumerate}[(a)]
\item{We prove the statement if $\varphi$ is injective or surjective. Then the claim follows by factoring $\varphi$ as $M \to \im \varphi \to N$. Assume that $\varphi$ is injective. Since $\tau(M)$ is given as a minimal submodule of $M$ that satisfies certain properties for every $\eta \in \Ass M \subseteq \Ass N$ we obtain that $\tau(M) \subseteq \tau(N)$ since $\tau(N)$ satisfies these properties as well.

Assume now that $\varphi$ is surjective. By Lemma \ref{ImageofFregularisFRegular} the image of $\tau(M)$ by $\varphi$ is $F$-regular. By the first part the natural inclusion $\varphi(\tau(M)) \to N$ induces a morphism of test modules. Hence, $\varphi(\tau(M)) \subseteq \tau(N)$.}
\item{Denote by $i$ the inclusion $M \to M \oplus N$. By (a) we have $i(\tau(M)) \subseteq \tau(M \oplus N)$ and similarly for $N$. In particular, $\tau(N) \oplus \tau(M) \subseteq \tau(M \oplus N)$. Note that $\Ass N \cup \Ass M = \Ass (M \oplus N)$ and that $H^0_\eta(N_\eta)$ is not nilpotent if and only if $\eta \in \Ass N$. Let $\eta$ be an associated prime of $M \oplus N$. Then we have a nil-isomorphism \begin{align*}H_\eta^0((\tau(M) \oplus \tau(N))_\eta) &= H_\eta^0(\tau(M)_\eta) \oplus H_\eta^0(\tau(N)_\eta) \\&\cong H^0_\eta(M_\eta) \oplus H^0_\eta(N_\eta) = H^0_\eta((M \oplus N)_\eta).\end{align*} Since $\tau(M \oplus N)$ is minimal with this property we obtain the other inclusion $\tau(M \oplus N) \subseteq \tau(M) \oplus \tau(N)$.}
 \end{enumerate}
\end{proof}

\begin{Bsp}
It is not true that $\varphi(\tau(M)) = \tau(N)$ if $\varphi: M \to N$ is surjective. Consider $M = k[x]$ with Cartier structure $\kappa x^{p-1}$ and $N = k[x]/(x)$ with $\varphi$ the canonical projection. Then $\tau(M) = (x)$ which is mapped to zero by $\varphi$. However, $\tau(N) = N$.
\end{Bsp}

\begin{Le}
\label{TauNilisoLemma}
Let $\varphi: M \to N$ be a nil-isomorphism of Cartier modules. Then $\varphi(\tau(M)) = \tau(N)$.
\end{Le}
\begin{proof}
By Proposition \ref{TestModulesAreFunctorial} (a) we have $\varphi(\tau(M)) \subseteq \tau(N)$. Moreover, we have for any associated prime $\eta$ of $M$ the following commutative diagram
\[\begin{xy}
   \xymatrix{ 0 \ar[r]& H^0_\eta(\varphi(M)_\eta) \ar[r] & H^0_\eta(N_\eta) \ar[r] & H^0_\eta(\coker \varphi_\eta) \\
0 \ar[r] & H^0_\eta(\varphi(\tau(M))_\eta) \ar[r] \ar[u] & H^0_\eta(\tau(N)_\eta) \ar[u]}
  \end{xy}
\]
where the vertical arrows are inclusions and nil-isomorphisms by the definition of the test module. Since $H^0_\eta(\coker \varphi_\eta)$ is nilpotent the inclusion $(H^0_\eta(\varphi(M)_\eta)) \subseteq H^0_\eta(N_\eta)$ is also a nil-isomorphism. We conclude that the inclusion $H^0_\eta(\varphi(\tau(M))_\eta) \to H^0_\eta(\tau(N)_\eta)$ is also a nil-isomorphism. By Lemma \ref{AssDescendsToCrystals} we have $\Ass M = \Ass N$ so that $\varphi(\tau(M))$ is a submodule for which the natural inclusion induces a nil-isomorphism $H^0_\eta(\varphi(\tau(M)_\eta) \subseteq H^0_\eta(\tau(N)_\eta) \subseteq H^0_\eta(N_\eta)$ at all associated primes of $N$. Since $\tau(N)$ is minimal with this property and $\varphi(\tau(M)) \subseteq \tau(N)$ the claim follows.
\end{proof}

\begin{Prop}
\label{TauFunctorOnCrys}
$\tau$ induces a functor on Cartier crystals.
\end{Prop}
\begin{proof}
If $\varphi: M \to N$ is a nil-isomorphism then $\ker \tau(\varphi) = \ker \varphi \cap \tau(M)$ is nilpotent. By Lemma \ref{TauNilisoLemma} above $\tau(\varphi)$ is surjective. Hence, $\tau(\varphi)$ is a nil-isomorphism and the claim follows from \cite[Proposition 2.2.4]{boecklepinktaucrystals}
\end{proof}

Let us recall, that for a closed immersion $i: \Spec R/I \to \Spec R$ the functor $i^\flat$ is simply the $I$-torsion functor given by $\Hom_R(R/I, -)$ considered as an $R/I$-module.
Next we list several basic properties of test modules (cf.\ \cite[Proposition 3.2]{blicklep-etestideale}).
\begin{Prop}
\label{TauFormalProperties}
Let $R$ be noetherian, $\mathcal{C}$ an $R$-Cartier algebra and $M$ a coherent $\mathcal{C}$-module. If $\tau(M)$ exists then the following hold:
\begin{enumerate}[(a)]
 \item{$\tau(M) = \tau(\underline{M})$.}
\item{If $S \subseteq R$ is a multiplicative set, then $\tau(S^{-1}M) = S^{-1}\tau(M)$.}
\item{If $\Supp(\underline{M}) \subseteq V(I)$ for some ideal $I \subseteq R$ then $\tau(i^\flat M) = i^\flat\tau(M)$, where $i \colon \Spec R/I \to \Spec R$ denotes the natural inclusion. In particular, this allows us to reduce to the case that $R$ is reduced and that $\Supp{\underline{M}} = \Spec R$.}
\item{If $i: \Spec R/I \to \Spec R$ is the natural inclusion then $i_\ast \tau(M) = \tau(i_\ast M)$.}
\end{enumerate}
\end{Prop}
\begin{proof}
Part (a) is immediate from the definition. We come to part (b). By Lemma \ref{AssForCrystals} (c) every $\eta \in \Ass S^{-1}M$ is also an associated prime of $\Ass M$. Hence, we have a nil-isomorphism $H^0_\eta(\tau(M)_\eta) \subseteq H^0_\eta(M_\eta)$ for every $\eta \in \Ass S^{-1} M$ which induces a nil-isomorphism $H^0_\eta(S^{-1}\tau(M)_\eta) \subseteq H^0_\eta(S^{-1}M_\eta)$. We still have to show that $S^{-1}\tau(M)$ is minimal with this property. So let $N \subseteq S^{-1} M$ be a Cartier submodule for which $H^0_\eta(N_\eta) \subseteq H^0_\eta(S^{-1}M_\eta)$ is a nil-isomorphism for each $\eta \in \Ass S^{-1}M$ and fix $\nu \in \Ass M$. Denote by $\varphi: M \to S^{-1}M$ the localization map. If $\nu \cap S = \varnothing$ then $\varphi^{-1}(N)_\nu = N_\nu$ and hence $H^0_\nu(\varphi^{-1}(N)_\nu) \subseteq H^0_\nu(M_\nu)$ is a nil-isomorphism. Otherwise, if $\nu \cap S \neq \varnothing$ then $(\ker \varphi)_\nu = M_\nu$ so that $\varphi^{-1}(N)_\nu = M_\nu$. In particular, $H^0_\nu(\varphi^{-1}(N)_\nu) = H^0_\nu(M_\nu)$.
 But $\tau(M)$ is minimal with this property so that $\tau(M) \subseteq \varphi^{-1}(N)$. This shows that $S^{-1}\tau(M) \subseteq S^{-1}\varphi^{-1}(N) = N$.

For part (c) we may replace $M$ by $\underline{M}$ and thus assume that $M$ is $F$-pure. The claim then follows from \cite[Lemmata 2.19, 2.20]{blicklep-etestideale}. Also note that $i^\flat M = \Hom_R(R/I, M)$ may be identified with $M$ considered as an $R/I$-module.

Finally, part (c) yields via adjunction for $(i_*,i^\flat)$ an inclusion $i_\ast \tau(M) \subseteq \tau(i_\ast M)$. On the other hand, one easily sees that $\tau(i_\ast M) \subseteq i_\ast \tau(M)$ by minimality of $\tau$.
\end{proof}

\begin{Bem}
Everything discussed in the section generalizes to the context of noetherian schemes. Indeed, since we only consider local cohomology with respect to a point all the $H^0_\eta$ factor through some open affine $\Spec A \subseteq X$. In particular, Proposition \ref{TauFormalProperties} (b) and uniqueness of the test module show that existence of $\tau$ may be checked on an open affine cover and that one obtains $\tau$ by gluing. Hence, we can always argue on local affine charts.
\end{Bem}

\section{Comparison to the previous definition}
\label{SectionComparisonOldDef}

We now show several comparison results of our notion of test module with the one introduced in \cite[Definition 3.1]{blicklep-etestideale}. For this we will denote the test module as introduced in \cite{blicklep-etestideale} by $\tau'(M)$ while we continue to use $\tau(M)$ as defined in Definition \ref{DefTestmodule}.

Recall that $\tau'(M)$ is defined as the minimal $\mathcal{C}$-submodule $N$ of $\underline{M}$ such that $N \subseteq \underline{M}$ yields an equality after \emph{localizing} at any \emph{minimal} prime of $\Supp \underline{M}$ (where, as usual, $\Supp $ denotes the support of the underlying module).

\begin{Prop}
\label{OldNewTestModuleRelation}
Let $M$ be a Cartier module. Then $\tau'(M) \subseteq \tau(M)$ with equality if all associated primes of $M$ are minimal.
\end{Prop}
\begin{proof}
We may assume that $M$ is $F$-pure. By Proposition \ref{TauFormalProperties} (c) and \cite[Proposition 3.2 (d), (f)]{blicklep-etestideale} we may further assume that $R$ is reduced and $\Supp M = \Spec R$. In particular, $R_\eta$ is a field for any minimal prime $\eta$ so that $H^0_\eta(M_\eta) = M_\eta$. If $\eta$ is a minimal prime of the underlying module $M$ then $M_\eta$ is not nilpotent by \cite[Lemma 1.4]{staeblerunitftestmoduln} so the minimal associated primes of the module coincide with those of the Cartier module. If the inclusion $N_\eta \subseteq M_\eta$ is a nil-isomorphism, then by definition the kernel and cokernel are nilpotent. But $M_\eta$ is $F$-pure so that it does not admit nilpotent quotients. This shows that the condition for $\tau(M)$ and $\tau'(M)$ at minimal primes of $\Supp M$ coincide. In particular, if all associated primes are minimal equality holds. Since $\tau'(M)$ is minimal with respect to a condition on all minimal primes of $\Supp M$ and $\tau(M)$
also satisfies these conditions we obtain, in general, an inclusion $\tau'(M) \subseteq \tau(M)$.
\end{proof}

\begin{Le}
\label{LocalizatonAtcCoincidesInclusionAfterMultbyc}
Let $M$ be a coherent $F$-pure Cartier module and $N$ a submodule. If $N_c = M_c$ for some $c$ then $cM \subseteq N$.
\end{Le}
\begin{proof}
Since $M/N$ is finitely generated there is $k \in \mathbb{N}$ such that $c^k (M/N) = 0$. But the quotient $M/N$ is $F$-pure since $M$ is $F$-pure so that \cite[Lemma 2.19]{blicklep-etestideale} tells us that we may choose $k = 1$.
\end{proof}

The following lemma is a weak version for $\tau$ of \cite[Theorem 3.11]{blicklep-etestideale} concerning test elements.

\begin{Le}
\label{WeakTestelementLemma}
Let $(M, \mathcal{C})$ be a Cartier module and assume that there is an $f$ not contained in any associated prime of $M$ and such that $M_f$ is $F$-regular. Then $\tau(M, \mathcal{C})$ is the $\mathcal{C}$-module generated by $f \underline{M}$.
\end{Le}
\begin{proof}
By Proposition \ref{TauFormalProperties} we may assume that $M$ is $F$-pure, $R$ is reduced and $\Supp M = \Spec R$. Let $N \subseteq M$ be a Cartier submodule such that $H^0_\eta(N_\eta) \subseteq H^0_\eta(M_\eta)$ is a nil-isomorphism for each $\eta \in \Ass M$. By $F$-regularity the inclusion $N_f \subseteq M_f$ is an equality and Lemma \ref{LocalizatonAtcCoincidesInclusionAfterMultbyc} yields $fM \subseteq N$.
It follows that $\mathcal{C} f M = \tau(M, \mathcal{C})$ as claimed since $\tau(M, \mathcal{C})$ is the minimal such $N$ and $H^0_\eta(\mathcal{C}fM_\eta) \subseteq H^0_\eta(M_\eta)$ induces a nil-isomorphism for each $\eta \in \Ass M$ since $f$ is not contained in any associated prime of $M$.
\end{proof}

\begin{Bem}
We note that in the context of Lemma \ref{WeakTestelementLemma} the $\mathcal{C}$-module generated by $f \underline{M}$ is also given by $N := \sum_{e \geq e_0} \mathcal{C}_e f \underline{M}$ for any $e_0 \geq 0$. One notes that $N$ is a Cartier module and that localizing at $f$ yields $N_f = \sum_{e \geq e_0} \mathcal{C}_e \underline{M}_f$. The inclusion $\mathcal{C}_+^h \subseteq \sum_{e \geq h} \mathcal{C}_h$ for any $h \geq 0$ and $F$-purity of $\underline{M}_f$ then imply that $N_f = \underline{M}_f$. By Lemma \ref{LocalizatonAtcCoincidesInclusionAfterMultbyc} $N$ then contains $f\underline{M}$ and since $N$ is a Cartier module also $\mathcal{C} f \underline{M}$.
\end{Bem}

\begin{Theo}
\label{FregularTausCoincide}
If $(M, \kappa)$ is $F$-regular in the sense of \cite[Definition 3.4]{blicklep-etestideale} (i.e.\, $\tau'(M, \kappa) = M$) and $\mathcal{C}$ is the Cartier algebra generated by $\{\kappa^e f^{\lceil tp^e\rceil}\,|\, e \geq 1\}$, where $f$ is a non zero-divisor on $M$ and $t \geq 0$ a real number, then $\tau'(M, \mathcal{C}) = \tau(M, \mathcal{C})$.
\end{Theo}
\begin{proof}
Since $f$ is a non zero-divisor it is not contained in any of the associated primes of $M$. By \cite[Theorem 3.11]{blicklep-etestideale} and since $\mathcal{C}_f = \kappa$ we obtain that $\tau'(M, \mathcal{C}) = \sum_{e \geq 1} \mathcal{C}_e f \underline{M}$. For the computation of $\tau(M, \kappa)$ we have by assumption that $\tau'(M, \kappa) = M$ and Proposition \ref{OldNewTestModuleRelation} thus yields $\tau(M, \kappa) = M$. We apply Lemma \ref{WeakTestelementLemma} and obtain likewise $\tau(M, \mathcal{C}) = \sum_{e \geq 1} \mathcal{C}_e f \underline{M}$.
\end{proof}

\begin{Bem}
Note that it was previously observed in \cite[Proposition 4.2]{staeblertestmodulnvilftrierung} that $\tau'$ for $F$-regular Cartier modules in the sense of \cite[Definition 3.4]{blicklep-etestideale} is functorial. We also remark that many results of \cite{staeblertestmodulnvilftrierung} and the main result of \cite{blicklestaeblerbernsteinsatocartier} were proved using the test module theory as developed in \cite{blicklep-etestideale} under the assumption that $(M, \kappa)$ is $F$-regular and that the Cartier algebra is of the form $\mathcal{C}_e = \kappa f^{\lceil tp^e\rceil} R$.

We also note that in the comparison of $\tau'$ with Stadnik's $V$-filtration (\cite{stadnikvfiltrationfcrystal}, \cite{staeblerunitftestmoduln}) the Cartier modules considered do not have non-minimal associated primes so that $\tau' = \tau$.
\end{Bem}

\section{Existence of test modules; test elements}
\label{SectionExistenceTestElements}
In this section we develop a theory of test elements similar to the classical case. They are a key technical ingredient in both existence results and computations. The main difference to the classical case is that we will have to deal with a sequence of test elements for a single given module. This is due to the fact that we may have inclusion relations between the associated primes $P_i$. Hence, given $c_i \notin P_i$ we can, in general, no longer find a single $c$ such that $D(c) \subseteq \bigcup_{i=1}^n D(c_i)$.

\begin{Def}
Let $M$ be a Cartier module with associated primes $\eta_1, \ldots, \eta_n$. We call $c_1, \ldots, c_n$ a \emph{sequence of test elements} if $c_i \notin \eta_i$ and if $\underline{H^0_{\eta_i}(M_{c_i})}$ is $F$-regular.
\end{Def}

Example \ref{FpurityNotPreservedByH^0} above shows that $H^0_{\eta_i}(M_{c_i})$ is not automatically $F$-pure if $M_{c_i}$ is $F$-pure. We now prove a partial analogue of \cite[Theorem 3.11]{blicklep-etestideale}. Namely, we show that if $M$ admits a sequence of test elements then $\tau(M)$ exists and we can express it somewhat explicitly using test elements. However, we can only prove a converse to this under the additional assumption that the test module exists for all submodules and (sub)quotients of $M$. This holds e.g.\ if $R$ is essentially of finite type over an $F$-finite field.

Before we proceed we formulate an easy lemma that is central to many computations. It tells us that, up to nilpotence, we may replace $H^0_\eta$ with $i^!$, where $i: \Spec R/\eta \to \Spec R$.

\begin{Le}
\label{LocalCohomShriekUpToNilpotence}
Let $M$ be a Cartier module and $I \subseteq R$ an ideal with associated closed immersion $i: \Spec R/I \to \Spec R$. Then $\underline{H^0_{I}(M)} = \underline{i_\ast i^!M}$.
\end{Le}
\begin{proof}
One may identify $i_\ast i^! M$ with $\{ m \in M \, \vert \, I m = 0\}$. This shows that we have an inclusion $i_\ast i^! M \subseteq H^0_I(M)$. On the other hand, since the annihilator of an $F$-pure module is radical (\cite[Lemma 2.19]{blicklep-etestideale}) we conclude that $\underline{H^0_I(M)} \subseteq i_\ast i^!M$. Applying $\underline{\phantom{M}}$ to the inclusions yields the desired equality.
\end{proof}

\begin{Bem}
\label{ShriekLocCohomAbuseOfNotation}
Since $i_\ast \underline{M} = \underline{i_\ast M}$ (cf.\ Proposition \ref{PushforwardPreservesNilpotence} below) and since $i^! i_\ast \cong \id$ induces an equivalence of categories we will frequently view $\underline{H^0_I(M)}$ as an $R/I$-module and identify it with $\underline{i^!M}$. Likewise we will also identify $\underline{i^!M}$ with $H^0_I(M)$ viewed as $R$-modules.
\end{Bem}

\begin{Theo}
\label{TestElementExistence}
Let $M$ be a Cartier module, if $M$ admits a sequence of test elements $c_1, \ldots, c_n$ then $\tau(M)$ exists and is given by \[\tau(M) = \sum_{i=1}^n \sum_{e \geq e_0} \mathcal{C}_e c_i \underline{H^0_{\eta_i}(M)} \quad \text{for any } e_0 \geq 0.\]
\end{Theo}
\begin{proof}
By Proposition \ref{TauFormalProperties} (a), (c) we may assume that $M$ is $F$-pure and that $\Supp M = \Spec R$.

Let us assume that we are given a sequence $c_1, \ldots, c_n$ of test elements. Let $N \subseteq M$ be a Cartier submodule such that for each $\eta_i \in \Ass M$ the inclusion $H^0_{\eta_i}(N_{\eta_i}) \subseteq H^0_{\eta_i}(M_{\eta_i})$ is a nil-isomorphism. We claim that $c_i \underline{H^0_{\eta_i}(M)} \subseteq N$ for each $i$. Once the claim is proven the implication follows since then also for any $e_0 \geq 0$ and every $i$ we have $ \sum_{e \geq e_0} \mathcal{C}_e c_i \underline{H^0_{\eta_i}(M)} \subseteq N$.

For each $i$ we have an inclusion $H^0_{\eta_i}(N_{c_i}) \subseteq H^0_{\eta_i}(M_{c_i})$. We claim that applying $\underline{\phantom{M}}$ on both sides yields an equality.
We certainly have an inclusion $\underline{H^0_{\eta_i}(N_{c_i})} \subseteq \underline{H^0_{\eta_i}(M_{c_i})}$ which is a nil-isomorphism. Moreover, $H^0_{\eta_i}(\underline{H^0_{\eta_i}(M_{c_i})}_{\eta_i}) = \underline{H^0_{\eta_i}(M_{\eta_i})}$ and similarly for $H^0_{\eta_i}(N_{c_i})$. Hence, localizing the inclusion at $\eta_i$ and applying $H^0_{\eta_i}$ yields by our assumption on $N$ and the $F$-regularity of
$\underline{H^0_{\eta_i}(M_{c_i})}$ the claimed equality. Both modules are finitely generated so that we get $c_i^t \underline{H^0_{\eta_i}(M)} \subseteq \underline{H^0_{\eta_i}(N)}$ for some $t \geq 0$. Applying Lemma \ref{LocalizatonAtcCoincidesInclusionAfterMultbyc} implies $c_i \underline{H^0_{\eta_i}(M)} \subseteq \underline{H^0_{\eta_i}(N)} \subseteq N$.
\end{proof}

\begin{Le}
\label{DoubleH0FpureCompatible}
Let $M$ be a Cartier module with associated prime $\eta$ and $I \subseteq \eta$ an ideal. Then $\underline{H^0_\eta(\underline{H^0_{I}(M)})} = \underline{H^0_\eta(M)}$.
\end{Le}
\begin{proof}
Since $H^0_\eta$ induces a functor on crystals the inclusion \[\underline{H^0_\eta(\underline{H^0_{I}(M)})} \subseteq \underline{H^0_\eta(H^0_{I}(M))} = \underline{H^0_\eta(M)}\] is a nil-isomorphism. By $F$-purity of $\underline{H^0_\eta(M)}$ it is therefore an equality.
\end{proof}

\begin{Theo}
\label{TauImpliesTestElements}
Let $M$ be a Cartier module on $\Spec R$, where $R$ is essentially of finite type over an $F$-finite field. Then there exists a sequence $c_1, \ldots, c_n$ of test elements.
\end{Theo}
\begin{proof}
By Proposition \ref{TauFormalProperties} (a), (c) we may assume that $M$ is $F$-pure and that $\Supp M = \Spec R$ with $R$ reduced. We find $c \in R$ such that $D(c) \cap \Ass M$ consists of all minimal primes of $\Supp M$. There exists $d \in R$ such that $\Ass M_{cd} = \{ \eta \}$. We claim that there is $f \in R$ such that $\underline{H^0_\eta(M_{cdf})}$ is $F$-regular. As $\eta$ is the only minimal prime of $M_{cd}$ we conclude that $R_{cd}$ is an integral domain. Hence, $H^0_\eta(M_{cdf}) = M_{cdf}$ and this is already $F$-pure. Finally, since the only associated primes of $M_{cd}$ are minimal the claim follows from Proposition \ref{OldNewTestModuleRelation}, the fact that $\tau(M_{cd}) = \tau(M)_{cd}$ (Proposition \ref{TauFormalProperties} (b)), \cite[Theorem 3.11]{blicklep-etestideale} and \cite[Theorem 4.13]{blicklep-etestideale}.

Note that the associated primes of $\underline{i^!M} = \underline{H^0_{(c)}(M)}$, where $i: \Spec R/(c) \to \Spec R$, are those that contain $c$, i.e.\ all associated primes of $M$ except the minimal primes. We find an element $c'$ that such that $D(c') \cap V(c) \cap \Ass \underline{i^!M}$ consists of the minimal associated primes of $i^!M$. Fix such a minimal associated prime $\nu$. By the above argument we find an element $u \notin \nu$ such that $\underline{H^0_\nu(\underline{i^!M})}_u$ is $F$-regular. By Lemma \ref{DoubleH0FpureCompatible} $\underline{H^0_\nu(\underline{H^0_{(c)}(M)})}_u = \underline{H^0_\nu(M)}_u$. We conclude that $u$ is a test element for $\nu$ and are done by descending induction.
\end{proof}

\begin{Bem}
In the proof of Theorem \ref{TauImpliesTestElements} the key issue is that, when applying the induction step, we need to know that $\tau$ exists for all sub-quotients of $M$ which only have minimal associated primes. By \cite[Theorem 3.9]{blicklep-etestideale} this is the case if and only if an $F$-pure $R$-Cartier module $M$ has finite length. This is also satisfied for an $F$-finite ring $R$ if the Cartier algebra is generically principal or if $\rk M \leq 1$. (cf. \cite[Proposition 3.15]{blicklep-etestideale} and the discussion before \cite[Lemma 3.12]{blicklep-etestideale}). We expect that test modules exist for finitely generated Cartier modules over any $F$-finite Ring, but this remains to be shown in this generality.
\end{Bem}

The following examples shows that given an associated prime $\eta$ of an $F$-regular $\mathcal{C}$-module $M$ then $\underline{H_\eta^0(M)}$ need not be $F$-regular.

\begin{Bsp}
Consider $R = \mathbb{F}_p[x,y]$ with $p \geq 3$ and the Cartier module $M = \omega_R \oplus \omega_R/(x)\omega_R$ with Cartier structure \[(a,b + (x)\omega_R) \mapsto (\kappa(xa), \kappa(x^{p-1}a) + \kappa((xy)^{p-1} (b + (x)\omega_R))),\] where $\kappa$ denotes the Cartier operator on $\omega_R$. Just as in Example \ref{FpurityNotPreservedByH^0} one verifies that $M$ is $F$-pure. Note that $x, y$ is a sequence of test elements for the associated primes $(0), (x)$. Indeed, $H^0_{(0)}(M)_x = \omega_{R_x}$ with Cartier structure $a \mapsto \kappa(xa)$ and $H^0_{(x)}(M)_y = (\omega_R/(x)\omega_R)_y$ with Cartier structure $b + (x)\omega_R \mapsto \kappa((xy)^{p-1}(b + x \omega_R))$. One easily checks that $\mathcal{C}_+ xH^0_{(0)}(M) = M$. Hence, $M$ is $F$-regular by Theorem \ref{TestElementExistence}. However, $\underline{H^0_{(x)}(M)} = \omega_R/(x) \omega_R$ is not $F$-regular. In particular, the constant sequence $1,1$ is not a sequence of test elements although $M$ is $F$-regular. This is in contrast to the case when one only considers minimal primes (cf. \cite[Proposition 5.2]{staeblertestmodulnvilftrierung}). This example also illustrates that $F$-regularity of $M$ does not imply $F$-regularity of $\underline{i^!M}$ which one might naively expect.
\end{Bsp}

\section{Properties of test modules}
\label{SectionStructuralResults}
In this section we fix an ideal $\mathfrak{a} \subseteq R$ and given a Cartier algebra $\mathcal{C}$ and a non-negative real number $t$ we consider the Cartier algebra $\mathcal{C}^{\mathfrak{a}^t}$ which is given by $\mathcal{C}^{\mathfrak{a}^t}_e = \mathcal{C}_e \mathfrak{a}^{\lceil t p^e \rceil}$ for $e \geq 1$. If $M$ is a $\mathcal{C}$-module then we denote the test module with respect to the Cartier algebra $\mathcal{C}^{\mathfrak{a}^t}$ by $\tau(M, \mathcal{C}^{\mathfrak{a}^t})$ or simply by $\tau(M, \mathfrak{a}^t)$.

Throughout this section we will assume that $R$ is essentially of finite type over an $F$-finite field (except for the results after Remark \ref{AssGradedRemark}). We will prove a Brian\c{c}on-Skoda Theorem as well as semi-continuity and discreteness results for our notion of test module, where we let $t$ vary for a fixed ideal $\mathfrak{a}$. The reason for our hypothesis on $R$ is twofold. For one, it ensures that test modules of finite Cartier modules exists which allows us to apply test element theory. Moreover, we use results of Blickle (\cite{blicklep-etestideale}) on gauge-boundedness of Cartier algebras in order to prove discreteness results.

We start by proving right-continuity. The argument is similar the one in \cite[Proposition 4.16]{blicklep-etestideale}

\begin{Prop}[Right-continuity]
\label{Rightcontinuity}
For an $R$-Cartier module $M$, $\mathfrak{a} \subseteq R$ an ideal and $t > 0$ a real number we have for all $\eps \geq 0$
\[ \tau(M, \mathfrak{a}^t) \supseteq \tau(M,\mathfrak{a}^{t + \eps})\] with equality for sufficiently small $\eps > 0$.
\end{Prop}
\begin{proof}
Using Proposition \ref{TauFormalProperties} (a), (c) we may replace $M$ by $\underline{M}_{\mathcal{C}^{\mathfrak{a}^t}}$ and assume that $\Supp M = \Spec R$ with $R$ reduced. We now want to argue that we find a sequence $c_1, \ldots, c_n$ which is a sequence of test elements for both $(M, \mathcal{C}^{\mathfrak{a}^t})$ and $(M, \mathcal{C}^{\mathfrak{a}^{t+\eps}})$.

If $M = 0$ then there is nothing to prove. So assume that $M \neq 0$. We claim that $\mathfrak{a}$ is not contained in the union of the minimal primes. Otherwise, it is contained in some minimal prime $\eta$ by prime avoidance. Then we have $\mathcal{C}^{\mathfrak{a}^t}_\eta = 0$ since $\mathfrak{a} \subseteq \eta$ is zero in $R_\eta$. Since $M_\eta$ still is $F$-pure we obtain a contradiction.

Fix a minimal prime $\eta$. We find $c$ such that $D(c) \cap \Ass M = \{\eta\}$ and such that $\underline{H^0_\eta(M)}_c$ is $F$-regular. Multiplying $c$ by $a$ for some $a \in \mathfrak{a} \setminus \eta$ we may further assume that $c \in \mathfrak{a}$. Hence, for all $t$ we have $\mathcal{C}^{\mathfrak{a}^t}_c = \mathcal{C}$. We conclude that $(\underline{H^0_\eta(M)}_c, \mathcal{C}^{\mathfrak{a}^{t + \eps}})$ is also $F$-regular.

If $\eta$ is a non-minimal associated prime then we find $c'$ such that $\Ass H^0_{c'}(M)$ admits $\eta$ as a minimal associated prime. We may again assume that $\Supp H^0_{c'}(M) = \Spec R$ and that $R$ is reduced. Using Lemma \ref{DoubleH0FpureCompatible} we have reduced to the previous situation.

By Theorem \ref{TestElementExistence} we conclude that for all $\eps \geq 0$ \[\tau(M, \mathfrak{a}^{t+\eps}) = \sum_\eta \mathcal{C}^{\mathfrak{a}^{t + \eps}} c_\eta H^0_{\eta}(M) \subseteq \sum_\eta \mathcal{C}^{\mathfrak{a}^{t}} c_\eta H^0_{\eta}(M) = \tau(M, \mathfrak{a}^t).   \]
For the reverse inclusion fix $E > 0$ and choose $\eps < \frac{1}{p^E}$. Then we have \[\tau(M, \mathfrak{a}^{t+\eps}) \supseteq \sum_{e=1}^E \sum_\eta \mathcal{C}_e \mathfrak{a}^{\lceil (t + \eps) p^e \rceil} c_\eta H^0_\eta(M) \supseteq \sum_{e=1}^E \sum_\eta \mathcal{C}_e \mathfrak{a}^{\lceil t p^e \rceil} c_\eta^2 H^0_\eta(M) = \tau(M, \mathfrak{a}^t),\] where the last equality follows since $c^2_\eta$ is a test element if $c_\eta$ is one and by choosing $E$ sufficiently large (each $\mathcal{C}^{\mathfrak{a}^t}_+ c_\eta H^0_\eta(M) = N$ is a finite $R$-module so if $n_i$ are generators of $N$ then there is $\varphi \in \bigoplus_{e=1}^{E_i} \mathcal{C}_e \mathfrak{a}^{\lceil tp^e \rceil}$ and $m \in H^0_\eta(M)$ such that $\varphi(c_\eta m) = n_i$. Taking the maximum over the $E_i$ and then the maximum over those for the finitely many $\eta$ we find an $E$ as desired).
\end{proof}

\begin{Prop}[Brian\c{c}on-Skoda]
For an $R$-Cartier module $M$, $\mathfrak{a} \subseteq R$ an ideal and $t \geq 0$ a real number we have
\[\mathfrak{a} \cdot \tau(M, \mathfrak{a}^{t-1}) \subseteq \tau(M, \mathfrak{a}^t), \] with equality if $t$ is greater or equal than the minimal number of generators of $\mathfrak{a}$.
\end{Prop}
\begin{proof}
First one shows as in Proposition \ref{Rightcontinuity} that $(M, \mathcal{C}^{\mathfrak{a}^t})$ and $(M, \mathcal{C}^{\mathfrak{a}^{t+1}})$ admit a common sequence of test elements. Then the argument proceeds just as in \cite[Theorem 4.21]{blicklep-etestideale}.
\end{proof}

In order to obtain a discreteness result we need to require that the pair $(M,\mathcal{C})$ is \emph{gauge bounded} in the sense of \cite[Definition 4.8]{blicklep-etestideale}. This is satisfied in many situations of interest. It holds in particular, if $\mathcal{C}$ is of the form $\kappa^{\mathfrak{a}^t}$ or more generally whenever $\mathcal{D}$ is a finitely generated Cartier algebra and $\mathcal{C} = \mathcal{D}^{\mathfrak{a}^t}$ (\cite[Proposition 4.15]{blicklep-etestideale}) and $M$ is any coherent module.

\begin{Prop}[Discreteness]
\label{Discretenss}
Let $(M, \mathcal{C})$ be a gauge bounded Cartier module, $\mathfrak{a}$ an ideal and $T > 0$. Then the set of test modules $\tau(M, \mathfrak{a}^t)$ for $0 \leq t \leq T$ is finite.
\end{Prop}
\begin{proof}
The proof of \cite[Theorem 4.18]{blicklep-etestideale} applies.
\end{proof}

Some of these results also generalize to the case of mixed test modules of the form $\tau(M, \mathfrak{a}_1^{t_1}, \ldots, \mathfrak{a}_n^{t_n})$. In this case the Cartier algebra is given in degree $e$ by $\mathcal{C}_e \mathfrak{a}_1^{\lceil t_1 p^e \rceil} \cdots \mathfrak{a}_n^{\lceil t_n p^e \rceil}$. In the following we will write $\mathbb{R}^n \ni t > 0$ if $t_i > 0$ for all $i$.

\begin{Prop}
Fix ideals $\mathfrak{a}_1, \ldots, \mathfrak{a}_m$. Given $0 (t_1, \ldots, t_m) \in \mathbb{R}^m \setminus \{0\}$ there is $(r_1, \ldots, r_m) > 0$ such that for all $0 \leq (\eps_1, \ldots, \eps_m) \leq r$ the mixed test ideals $\tau(M, \mathfrak{a}_1^{t_1} \cdots \mathfrak{a}_n^{t_n})$ and $\tau(M, \mathfrak{a}_1^{t_1+\eps_1} \cdots \mathfrak{a}_n^{t_n+\eps_n})$ coincide.
\end{Prop}
\begin{proof}
The proof of Proposition \ref{Rightcontinuity} applies. One only needs to replace $c^2_\eta$ by $c^{n+1}_\eta$ in the last computation.
\end{proof}

Similarly one can prove a Brian\c{c}on-Skoda theorem for mixed test modules:

\begin{Prop}
 For an $R$-Cartier module $M$, and a sequence of Ideals $\mathfrak{a}_1, \dots ,\mathfrak{a}_n$ in $R$, and $0 \leq t_1,\ldots,t_n \in \mathbb{R}^n$ with $t_i \geq 1$ for some fixed $i$ we have
\[\mathfrak{a}_i \cdot \tau(M, \mathfrak{a}_1^{t_1} \cdots \mathfrak{a}_i^{t_i-1} \cdots \mathfrak{a}_n^{t_n}) \subseteq \tau(M, \mathfrak{a}_1^{t_1} \cdots \mathfrak{a}_n^{t_n}), \] with equality if $t_i$ is greater or equal than the minimal number of generators of $\mathfrak{a}_i$.
\end{Prop}

We do not know how to prove a discreteness result in the mixed test module case unless the field is finite. The issue is that the set of test modules $\tau(M, \mathfrak{a}_1^{t_1} \cdots \mathfrak{a}_n^{t_n})$ is in general not totally ordered.

\begin{Bem}
\label{AssGradedRemark}
An important special case is where the algebra $\mathcal{C}$ is generated by a single element $\kappa$. Given a principal ideal $\mathfrak{a} = (f)$ one can then form the associated graded of the test module filtration \[Gr^t(M, \mathfrak{a}) = \tau(M, \mathfrak{a}^{t- \eps})/\tau(M, \mathfrak{a}^t).\] This filtration again carries a Cartier module structure given by $\kappa f^{\lceil t (p-1) \rceil}$. Passing to crystals we are in a situation where we have an equivalence with constructible $p$-torsion sheaves (we refer to \cite{staeblertestmodulnvilftrierung} for a more detailed treatment).
\end{Bem}

One also has inclusion results similar to Proposition \ref{Rightcontinuity} in a quite general setting for the mixed test modules with an analogous proof henceforth omitted.

We now assume that $R$ is any noetherian ring. In the following we will denote by $\Ass_\mathcal{C}(M)$ the associated primes of the $\mathcal{C}$-module $M$.

\begin{Le}
\label{AssForDifferentAlgebras}
Let $\mathcal{C}' \subseteq \mathcal{C}$ be Cartier algebras and $N \subseteq M$, where $N$ is a $\mathcal{C}'$-module and $M$ is a $\mathcal{C}$-module. Then $\Ass_{\mathcal{C}'} N \subseteq \Ass_{\mathcal{C}} M$.
\end{Le}
\begin{proof}
If $H^0_\eta(M)_\eta$ is not $\mathcal{C}'$-nilpotent then a fortiori it is not $\mathcal{C}$-nilpotent. Hence, $\Ass_{\mathcal{C}'} M \subseteq \Ass_{\mathcal{C}} M$. Using Lemma \ref{AssForCrystals} (a) the claim follows.
\end{proof}

\begin{Le}
\label{Nilisointersection}
Let $A, B \subseteq M$ be Cartier submodules. If $B \subseteq M$ is a nil-isomorphism then $A \cap B \subseteq A$ is also a nil-isomorphism.
\end{Le}
\begin{proof}
The cokernel $A/(B\cap A)$ is contained in $M/B$ and the latter is nilpotent by assumption.
\end{proof}

\begin{Prop}
\label{TauForDifferentAlgebras}
Let $M$ be a $\mathcal{C}$-module and $\mathcal{C}' \subseteq \mathcal{C}$ then $\tau(M, \mathcal{C}') \subseteq \tau(M, \mathcal{C})$.
\end{Prop}
\begin{proof}
Note that $\underline{M}_{\mathcal{C}'} \subseteq \underline{M}_{\mathcal{C}}$. Assume now that $N \subseteq \underline{M}_{\mathcal{C}}$ is a $\mathcal{C}$-submodule such that $H^0_\eta(N)_\eta \subseteq H^0_\eta(\underline{M}_{\mathcal{C}})_\eta$ is a nil-isomorphism for all $\eta \in \Ass M$. It is, in particular, a nil-isomorphism of $\mathcal{C}'$-modules and we conclude with Lemma \ref{Nilisointersection} that $\underline{M}_{\mathcal{C}'} \cap N \subseteq \underline{M}_{\mathcal{C}'}$ induces a nil-isomorphism after localizing and applying $H^0_\eta$. From Lemma \ref{AssForDifferentAlgebras} and the fact that the test module is the minimal submodule $\tau$ such that $H^0_\eta(\tau)_\eta \subseteq H^0(\underline{M}_{\mathcal{C}'})_\eta$ is a nil-isomorphism we conclude that $\tau(M, \mathcal{C}') \subseteq \tau(M, \mathcal{C})$ as claimed.
\end{proof}

\begin{Ko}
Let $\mathcal{C}' \subseteq \mathcal{C}$ Cartier algebras and $N \subseteq M$, where $N$ is a $\mathcal{C}'$-module and $M$ is a $\mathcal{C}$-module. Then $\tau(N, \mathcal{C}') \subseteq \tau(M, \mathcal{C})$.
\end{Ko}
\begin{proof}
From Proposition \ref{TauForDifferentAlgebras} we get $\tau(M, \mathcal{C}') \subseteq \tau(M, \mathcal{C})$. By functoriality (Proposition \ref{TestModulesAreFunctorial}) we have $\tau(N, \mathcal{C}') \subseteq \tau(M, \mathcal{C}')$.
\end{proof}

We note the following lemma. It asserts that existence of $\tau$ may always be reduced to a question about finitely generated Cartier algebras.

\begin{Le}
\label{TauExistenceFinitelyGeneratedC}
Let $M$ be an $F$-pure coherent $\mathcal{C}$-module. Then there is a finitely generated Cartier algebra $\mathcal{C}' \subseteq \mathcal{C}$ such that $\underline{H^0_\eta(M)}_\mathcal{C} = \underline{H^0_\eta(M)}_{\mathcal{C}'}$ for every associated prime $\eta$ of $M$ and such that $\underline{M}_\mathcal{C} = \underline{M}_{\mathcal{C}'}$. For any such algebra we have $\mathcal{C} \tau(M, \mathcal{C}') = \tau(M, \mathcal{C})$.
\end{Le}
\begin{proof}
Since $M$ is coherent $\Ass_\mathcal{C}(M)$ is finite. Moreover, each of the $\mathcal{C}$-modules $H^0_\eta(M)$ (for $\eta \in \Ass_\mathcal{C}(M)$) is also coherent. Hence, we find finitely many (homogeneous) elements $\kappa_e \in \mathcal{C}_e$ such that $\sum_e \kappa_e(M) = M$. Fixing $\eta$ we find finitely many $\kappa'_e$ such that $\sum_e \kappa'_e \underline{H^0_\eta(M)}_\mathcal{C} = \underline{H^0_\eta(M)}_\mathcal{C}$. Denote the algebra generated by the $\kappa'_e$ and $\kappa_e$ by $\mathcal{C}'$. Then we have \[\underline{H^0_\eta(M)}_\mathcal{C} = \mathcal{C}'^e_+ \underline{H^0_\eta(M)}_\mathcal{C} \subseteq \mathcal{C}'^e_+ H^0_\eta(M) \subseteq \mathcal{C}^e_+ H^0_\eta(M) = \underline{H^0_\eta(M)}_\mathcal{C}.\]

Repeating this for all $\eta$ and enlarging $\mathcal{C}'$ we may assume that $\mathcal{C}'_+ M = M$ and $\underline{H^0_\eta(M)}_{\mathcal{C}'} = \underline{H^0_\eta(M)}_\mathcal{C}$.

Let now $N = \tau(M, \mathcal{C}')$. That is $N$ is the minimal $\mathcal{C}'$-submodule of $M$ for which $H^0_\eta(N)_\eta \subseteq H^0_\eta(M)_\eta$ is a nil-isomorphism for every $\eta$. In particular, we have \begin{equation}\label{asdf} \mathcal{C}^e_+ H^0_\eta(M)_\eta = \mathcal{C}'^e_+ H^0_\eta(M)_\eta \subseteq H^0_\eta(N)_\eta \end{equation} for all $e \gg 0$ and all $\eta$.

Note that one has $\mathcal{C}H^0_\eta(N) \subseteq H^0_\eta(\mathcal{C} N)$. Indeed, the left hand side is the $\mathcal{C}$-module generated by the $\eta^k$-torsion of $N$ while the right-hand side is a $\mathcal{C}$-module which contains the $\eta^k$-torsion of $N$ since $N \subseteq \mathcal{C} N$.

This observation shows together with (\ref{asdf}) that we have inclusions \[\mathcal{C}^e_+H^0_\eta(M)_\eta = \mathcal{C}_+ \mathcal{C}^e_+H^0_\eta(M)_\eta \subseteq \mathcal{C}_+ H^0_\eta(N)_\eta \subseteq \mathcal{C} H^0_\eta(N)_\eta \subseteq H^0_\eta(\mathcal{C} N)_\eta.\]
This shows that we have a nil-isomorphism $H^0_\eta(\mathcal{C} N)_\eta \subseteq H^0_\eta(M)_\eta$ for each associated prime $\eta$. Recall, that $\tau(M, \mathcal{C})$ is minimal with this property. Moreover, we have $N \subseteq \tau(M, \mathcal{C})$ by Proposition \ref{TauForDifferentAlgebras} and hence $\mathcal{C} N \subseteq \tau(M, \mathcal{C})$. By minimality we conclude that $\tau(M, \mathcal{C}) = \mathcal{C}N$.
\end{proof}

\section{Twisted inverse images and pushforwards for Cartier modules}
\label{SectionCartieralgebraPullbackPushforward}
The purpose of this section is to introduce the twisted inverse image and pushforward functors for Cartier modules. We show that these functors are compatible with nil-isomorphisms and thus induce functors on the corresponding categories of Cartier crystals.

This extends some results of \cite{blickleboecklecartiercrystals}, where the case of a single structural morphism $\kappa: F_\ast M \to M$ was treated.

We recall first the definition of a Cartier algebra for a scheme. It generalizes the one for a ring (cf.\ Definition \ref{CartierAlgebraDefRing}) in a straightforward way.
\begin{Def}
An $\mathbb{N}$-graded sheaf of rings $\mathcal{C}_X$ on a scheme $X$ is a \emph{Cartier algebra over $X$} if $\mathcal{C}_X$ is a quasi-coherent $\mathcal{O}_X$-bimodule with $r \cdot \kappa = \kappa \cdot r^{p^e}$ for local sections $r, \kappa$ with $\kappa$ homogeneous of degree $e$. We further require that the homogeneous elements of degree zero of $\mathcal{C}_X$ coincide with $\mathcal{O}_X$. A \emph{morphism of Cartier algebras} is a ring homomorphism $\mathcal{C}_1 \to \mathcal{C}_2$ that is $\mathcal{O}_X$-linear for the bimodule structures.
\end{Def}

As in the case of $X = \Spec R$ a $\mathcal{C}$-module (or Cartier module) will always denote a \emph{left} $\mathcal{C}$-module. We will only consider Cartier modules which are (quasi-)coherent as $\mathcal{O}_X$-modules.

\begin{Le}
\label{Cartiermodulestructureviahomomorphism}
The structure of a $\mathcal{C}_X$-module on a sheaf of sets $M$ is equivalent to a coherent $\mathcal{O}_X$-module structure on $M$ together with a graded homomorphism of rings $\Xi: \mathcal{C}_X \to \bigoplus_{e \geq 0} \mathcal{H}om_{\mathcal{O}_X}(F^e_\ast M, M)$, here the multiplication on the right hand side is defined as $\varphi \cdot \psi = \varphi \circ F_\ast^e \psi$, if $\varphi$ is homogeneous of degree $e$.
\end{Le}
\begin{proof}
The correspondence is given by defining $\kappa_e \cdot m = \Xi(\kappa_e)(m)$ for $\kappa_e \in \mathcal{C}_X$ homogeneous of degree $e$.
\end{proof}

\begin{Prop}
\label{PullbackCartierRingStructure}
Let $f: X \to Y$ be a morphism of schemes and $\mathcal{C}_Y$ an $\mathcal{O}_Y$-Cartier algebra. Then $\mathcal{C}_X = f^{-1} \mathcal{C}_Y \otimes_{f^{-1} \mathcal{O}_Y} \mathcal{O}_X$ is an $\mathcal{O}_X$-Cartier algebra if we define its left $\mathcal{O}_X$-module structure for local sections via $s (\kappa \otimes t)= \kappa \otimes s^{p^e} t$, where $\kappa \in \mathcal{C}_e$, and its multiplication by $(\kappa \otimes s) \cdot (\kappa' \otimes t)= \kappa \kappa' \otimes s^{p^{e'}} t$, where $\kappa'$ is homogeneous of degree $e'$.
\end{Prop}
\begin{proof}
It is enough to show that the presheaf $f^{-1} \mathcal{C}_Y \otimes_{f^{-1}\mathcal{O}_Y} \mathcal{O}_X$ is a Cartier algebra, since sheafification preserves both ring and module structures. For any open $U \subseteq X$ we have $f^{-1}\mathcal{O}_Y(U) = \colim_{f(U) \subseteq V} \mathcal{O}_Y(V)$ and similarly for $f^{-1} \mathcal{C}_Y(U)$. Since both are filtered colimits we can endow $\mathcal{O}_X$ with the same colimit structure (with transition maps the identity) and apply \cite[Lemma 3.4]{blicklestaeblerbernsteinsatocartier} to obtain an isomorphism \[(f^{-1} \mathcal{C}_Y \otimes_{f^{-1}\mathcal{O}_Y} \mathcal{O}_X)(V) \cong \colim_{f(U) \subseteq V} \mathcal{C}_Y(V) \otimes_{\mathcal{O}_Y(V)} \mathcal{O}_X(U).\]

We have reduced to showing: If $R \to S$ is a ring homomorphism and $\mathcal{C}_R$ is a Cartier algebra then $\mathcal{C}_R \otimes_R S$ with multiplication and bimodule structure as defined above is an $S$-Cartier algebra. Since $S$ is commutative and of characteristic $p > 0$ the left-module structure is well-defined. Next, one verifies that the map \begin{align*}&(\mathcal{C}_R \times S) \times (\mathcal{C}_R \times S) \longrightarrow \mathcal{C}_R \otimes_R S,\\& ((\sum_e \varphi_e, s), (\sum_d \psi_d, s')) \longmapsto \sum_{e, d} \varphi_e \circ \psi_d \otimes s^{p^d} s'\end{align*} induces the desired multiplication map $\sum_e \varphi_e \otimes s \cdot \sum_d \psi_d \otimes s' = \sum_{e, d} \varphi_e \psi_d \otimes s^{p^d} s'$.
An explicit computation shows that this defines a ring structure on $\mathcal{C}_S$.
\end{proof}

\begin{Def}
If $f: X \to Y$ is a morphism and $\mathcal{C}_Y$ a Cartier algebra on $Y$ then we will denote $\mathcal{C}_X$ as constructed in Proposition \ref{PullbackCartierRingStructure} by $f^\ast \mathcal{C}_Y$ and refer to it as the \emph{pullback of $\mathcal{C}_Y$ (along $f$)}.
\end{Def}

Let $f: X \to Y$ be a morphism of noetherian schemes and $M$ a quasi-coherent $\mathcal{O}_Y$-module. If $f$ is finite then the twisted inverse image $f^! M$ is $\bar{f}^\ast \mathcal{H}om_{\mathcal{O}_Y}(f_\ast \mathcal{O}_X, M)$, where $\bar{f}: (X, \mathcal{O}_X) \to (Y, f_\ast \mathcal{O}_X)$ is the canonical flat map of ringed spaces. If $f$ is essentially smooth then the twisted inverse image $f^!M$ is given by $\omega_f \otimes f^\ast M$. In particular, if $f$ is essentially \'etale then $f^! M = f^\ast M$. Furthermore, in the essentially \'etale case one has an isomorphism $f^\ast F_\ast^e M \to F_\ast^e f^\ast M$ which locally is given by $F_\ast^e M \otimes_R S \to F_\ast^e(M \otimes_R S), m \otimes s \mapsto m \otimes s^{p^e}$ (cf.\ \cite[Lemma 2.2.1]{blickleboecklecartiercrystals}).

\begin{Theo}
\label{UpperShriekCartierStructure}
Let $f: X \to Y$ be a morphism of noetherian schemes and $M$ a $\mathcal{C}_Y$-module. Let $\kappa \otimes s \in {\mathcal{C}_X}$ be a local section homogeneous of degree $e$. We define a $\mathcal{C}_X$ module structure on $f^!M$ in the following cases.
\begin{enumerate}[(a)]
 \item For finite $f$ and a local section $\varphi \in f^!M$ we set $(\kappa \otimes s) \cdot \varphi = \kappa \circ F_\ast^e( \varphi \circ \mu_s) \circ F^e_X$, where $\mu_s$ is multiplication by $S$.
\item For essentially \'etale $f$: If $U \subseteq Y$ open affine and $U' \subseteq f^{-1}(U)$ open affine and $m \otimes t^{p^e} \in M(U)$ a local section we define a Cartier structure via $(\kappa \otimes s^{p^e}) \cdot m \otimes t^{p^e} = \kappa(m) \otimes ts$.
\item For $f$ smooth: If $\Spec R = U \subseteq Y$ is open affine and $U' \subseteq f^{-1}(U)$ is such that $U' \to U$ factors as $U' \to \mathbb{A}^n_R \to \Spec R$, where the first map is \'etale, then we define a $\mathcal{C}_{R[x]}$-module structure on $g^! M$, where $g: \Spec R[x] \to \Spec R$, by $\kappa \otimes \sum_i a_i x^i \cdot dx \otimes r \otimes m = \sum_i x^{\frac{i+1}{p^e} -1} dx \otimes 1 \otimes \kappa (r a_i m)$, where the expression $\frac{i+1}{p^e} -1 $ is understood to be zero whenever $\frac{i+1}{p^e}$ it is not an integer.
\end{enumerate}
\end{Theo}
\begin{proof}
We have to show in each case that the assignment \[\mathcal{C}_S \to \bigoplus_{e \geq 0} \Hom_S(F_\ast^e f^! M\vert_{U'}, f^!M\vert_{U'}),\] where $U' = \Spec  S \subseteq f^{-1}(U)$ and $\Spec R = U \subseteq Y$ are open and affine, is a ring homomorphism and that we can glue these homomorphisms to obtain a homomorphism $\mathcal{C}_X \to \mathcal{H}om(F_\ast^e f^! M, f^!M)$. By Lemma \ref{Cartiermodulestructureviahomomorphism} the claim then follows.
\begin{enumerate}[(a)]
 \item One readily verifies that the assignment is a ring homomorphism. % First of all, one checks that the map is well-defined. If $\kappa \otimes s$ and $\kappa' \otimes s'$ are local sections with homogeneous $\kappa$ of degree $e$ and $\kappa'$ homogeneous of degree $e'$ then given $\varphi \in f^!M$, one computes \[\kappa \kappa' \otimes s^{p^{e'}} s'(F_\ast^{e+e'} \varphi) = \kappa \kappa' \circ F_\ast^{e+e'} (\varphi \mu_{s^{p^{e'}}} \mu_{s'}) \circ F^{e+e'},\] where by abuse of notation we consider $\kappa$ as a map from $F_\ast^e M \to M$. On the other hand we obtain \[(\kappa \otimes s)(F_\ast^e(\kappa' \otimes s'(F_\ast^{e'}\varphi))) = \kappa \circ F_\ast^e(\kappa' \circ F_\ast^{e'}( \varphi \circ \mu_s) \circ F^{e'} \circ \mu_s \circ F^e)= \kappa \kappa' F_\ast^{e'}( \varphi \mu_s \mu_{s^{p^{e'}}}) \circ F^{e'+e}.\]
Gluing follows from the fact that both $\mathcal{C}_R$ and $\mathcal{O}_Y$ are quasi-coherent and that we have $f^\ast \mathcal{C}_R\vert_{U'} = (\mathcal{C}_Y(U) \otimes_{\mathcal{O}_Y(U)} \otimes \mathcal{O}_X(U))^{\sim}$. The map defining the module structure in turn is induced from a map $C_R \times S \to \bigoplus_e \Hom_S(F_\ast^e f^!M, f^!M)$ and this map clearly glues.% the maps are induced by a universal property from $C_R \times S$. The map on $C_R \times S$ clearly glues since we are dealing with quasi-coherents (and the actual construction of the map). Hence, we get the desired induced map which by the part that is not commented is all that we need.
\item{One checks that the map \[\mathcal{C}_{e, R} \times S \to \Hom_S(F_\ast^e f^!M, f^!M),  (\kappa, s) \mapsto \varphi_\kappa \circ F^e_\ast (id \otimes \mu_s)\] is bilinear, where $\varphi_\kappa$ is the map $F_\ast^e(M \otimes_R S) \to M \otimes_R S$ induced by the image of $\kappa \in \mathcal{C}_R$ in $\Hom_R(F_\ast^eM, M)$.

Hence, we obtain the desired map $\mathcal{C}_S \to \bigoplus_{e\geq0} \Hom_S(F_\ast^e f^!M, f^!M)$. Since this map is additive and since we can write any $s \in S$ as $\sum_i r_i s_i^{p^e}$ for suitable $r_i \in R$ and any $e \geq 1$ we only have to verify multiplicativity for homogeneous elements of the form $\kappa \otimes s^{p^a}, \kappa' \otimes s'^{p^b}$. This is a straightforward computation which we leave to the reader.

Note that we have natural maps of rings $\mathcal{C}_R \to \mathcal{C}_S$ and $S \to \mathcal{C}_S$. %(induced by $R \to S$ tensored with $\mathcal{C}_R \otimes_R$ and $R \to \mathcal{C}_R$ tensored with $\otimes_R S$})% Indeed.
Moreover, $f^!M$ carries an $S$-module structure and the image of $M \to f_\ast f^! M$ is a $\mathcal{C}_R$-module. The $\mathcal{C}_S$-module structure defined above is uniquely determined by the fact that it extends these module structures. In particular, the $\mathcal{C}_S$-module structure is compatible on intersections and thus glues.}
\item{In order to show that $\mathcal{C}_{U'} \to \bigoplus_{e \geq 0}\Hom(F_\ast^e f^! M\vert_U,f^! M\vert_U)$ is a ring homomorphism it suffices by (b) to show that the map \[\mathcal{C}_{\mathbb{A}^1_U} \to \bigoplus_{e \geq 0} \Hom_{R[x]}(F_\ast^e g^! M\vert_U,g^! M\vert_U)\] is a ring homomorphism. This in turn is easily verified.% Well-definedens is a trivial computation. Next, since this assignment clearly commutes with addition we may assume that elements are of the form $\kappa \otimes a x^i$ and $\kappa' \otimes b x^j$. Then one computes \[(\kappa \otimes a x^i) (\kappa' \otimes bx^j \cdot dx \otimes r \otimes m) = \kappa a x^i \cdot x^{\frac{j+1}{p^{e'}} - 1} dx \otimes 1 \otimes \kappa'(rbm) = x^{\frac{i + \frac{j+1}{p^{e'}}}{p^e} -1} dx \otimes 1 \otimes \kappa(a \kappa'(rbm))\] and \[\kappa \kappa' \otimes a^{p^{e'}} bx^{ip^{e'}} x^j \cdot dx \otimes r \otimes m = x^{\frac{ip^{e'} + j + 1}{p^{e +e'}} -1 }dx \otimes \otimes \kappa \kappa' (r a^{p^{e'}} b m) \]

Note that by \cite[Lemma 3.1 and the preceding discussion]{staeblerunitftestmoduln} our local definition of the module structure is functorial in the sense that we have an $R$-linear map \[\mathcal{C}_R \to \bigoplus_{e \geq 0} \Hom_R(F^e_\ast M , M) \to \bigoplus_{e \geq 0} f_\ast \Hom_S(F^e_\ast f^!M, f^!M)\] which uniquely corresponds to the $S$-linear map \[\mathcal{C}_S \to \bigoplus_{e \geq 0} \Hom_S(F_\ast^e f^!M, f^!M)\] that defines the module structure. %Compatibility with Scalars follows from functoriality, additivity is directly verified from 3.1 or 3.4 in \cite{staeblerunitftestmoduln}
Since the $R$-linear map glues we can also glue our module structure.} %note that the correspondence is just the adjucntion Hom_R(M, N) \cong Hom_S(M \otimes S,N) where N is an $S$-module.
\end{enumerate}
\end{proof}

\begin{Bem}
\begin{enumerate}[(a)]
 \item In defining a $\mathcal{C}_S$-module structure for $f$ \'etale on $f^!M$ one might think that it is induced by tensoring the map $\mathcal{C}_R \to \bigoplus_{e \geq 0} \Hom_R(F_\ast^e M , M)$ with $S$ and making the natural identification \[\Hom_R(F_\ast^e M, M) \otimes_R S \cong \Hom_S(F_\ast^e f^!M, f^!M).\] However, this is not the case. Indeed, with this identification the element $\kappa \otimes s^p$ corresponds to the map $[m \otimes t^p \mapsto \kappa(m) \otimes t s^p]$, while first applying $\id \otimes s^p$ and then $\kappa \otimes 1$ yields the map $[m \otimes t^p \mapsto \kappa(m) \otimes st]$.
\item One expects that if $f: X \to Y$ can be factored into a composition of finite and smooth morphisms, then one similarly has a functor $f^!$ from Cartier modules on $Y$ to Cartier modules on $X$. However, we do not verify the required compatibilities here (cf.\ also \cite[III.8]{hartshorneresidues}).
\end{enumerate}
\end{Bem}

The following is an immediate generalization of corresponding adjointness results shown in \cite{blickleboecklecartiercrystals} in the case of a single Cartier linear operator.
\begin{Prop}\label{adjunctionsoffiniteandetale}
Let $f \colon X \to Y$ be a morphism of noetherian schemes.
\begin{enumerate}
\item If $f$ is finite, then $(f_*,f^!)$ form an adjoint pair of functors of Cartier modules (resp. Cartier crystals).
\item If $f$ is essentially \'etale, the $(f^*,f_*)$ form an adjoint pair of functors of Cartier modules (resp. Cartier crystals).
\end{enumerate}
\end{Prop}
\begin{proof}
The respective units of adjunctions are obtained from the adjunctions of the underlying quasi-coherent sheaves. To check that these are Cartier linear for a Cartier algebra $\mathcal{C}$ comes down to checking this for every homogeneous element of $\mathcal{C}$, but this is precisely the case that is verified in \cite[Sections 3.3 and 3.4]{blickleboecklecartiercrystals}.  That these adjunctions on the level of Cartier modules induce such for the corresponding Cartier crystals is immediate.
\end{proof}

If $f: X \to Y$ is finite \'etale then we have defined a $\mathcal{C}_X$-module structure on $f^! M$ in two ways. These two module structures coincide:

\begin{Le}
\label{TracePthPowers}
Let $f: \Spec S \to \Spec R$ be finite \'etale. Then $Tr(s^p) = Tr(s)^p$.
\end{Le}
\begin{proof}
This is local so that we may assume that $S$ is a free $R$-module. Fix a basis $s_1, \ldots, s_n$. The isomorphism $F_\ast R \otimes_R S \to F_\ast S, r \otimes s \mapsto r s^p$ shows $s_1^p, \ldots, s_n^p$ is also a basis for $S$. One readily deduces the claim. %Explicitly, if $s s_i = \sum_j r_{ij} s_j$ then $(s s_i)^p = \sum_j r_j s_j^p$ since $s_j^p$ are also a basis we obtain that Tr(s^p) = sum_i r_{ii}^p = (\sum_i r_{ii})^p = Tr(s)^p
\end{proof}

\begin{Prop}
If $f: X \to Y$ is finite and \'etale and $M$ a $\mathcal{C}_Y$-module then one has an isomorphism of $\mathcal{C}_X$-modules $f^\ast M \to \bar{f}^\ast \mathcal{H}om_{\mathcal{O}_Y}(f_\ast \mathcal{O}_X, M)$.
\end{Prop}
\begin{proof}
This is local so that we may assume that $f: \Spec S \to \Spec R$. Then one has an isomorphism of $S$-modules $M \otimes_R S \to \Hom_R(S, M), m \otimes s^{p^e} \mapsto [t \mapsto Tr(s^{p^e}t) \cdot m]$, where $Tr: S \to R$ is the ordinary trace map (cf. e.g.\ \cite[1.4, 1.2 and Proposition 6.9]{lenstragaloistheory}). We need to verify that this is compatible with Cartier structures. So let $\kappa \otimes u^{p^e}$ be homogeneous of degree $e$. Then letting $\kappa \otimes u^{p^e}$ act on $t \mapsto Tr(s^{p^e} t) \cdot m$ we obtain according to Proposition \ref{UpperShriekCartierStructure} (a) and Lemma \ref{TracePthPowers} $\kappa (Tr((stu)^{p^e}) m) = Tr(stu) \kappa(m)$.

The other way around, first applying $\kappa \otimes u^{p^e}$ to $m \otimes s^{p^e}$ yields $\kappa(m) \otimes u s$ which is mapped to $t \mapsto Tr(stu) \kappa(m)$ by the above isomorphism.
\end{proof}

\begin{Ko}
\label{Nonnilpotencepreservedbyshriek}
Let $f: X \to Y$ be a finite or smooth morphism of noetherian schemes. If $M$ is a nilpotent $\mathcal{C}_Y$-module, then $f^!M$ is also nilpotent. In particular, $f^!$ preserves nil-isomorphisms and induces a functor on crystals.
\end{Ko}
\begin{proof}
This is a local issue so that we may assume that $Y = \Spec R$ and $X = \Spec S$ are both affine. The nilpotence of $M$ is equivalent to the fact that every homogeneous element of $\mathcal{C}_{R +}^e$ operates trivially on $N$. A homogeneous element of $\mathcal{C}_{S +}^e$ is a sum of elements $\kappa \otimes s$ with $\kappa \in \mathcal{C}_{R +}^e$. Let us assume that $f$ is finite. If $\kappa$ operates as zero on $M$ then clearly $(\kappa \otimes s) \varphi = 0$ for any $\varphi \in f^!M$.

The smooth case follows from a similar computation.
\end{proof}

\begin{Le}
\label{PushforwardCartierstructure}
Let $f: X \to Y$ be a morphism of noetherian schemes and let $\mathcal{C}_Y$ be a Cartier algebra on $Y$. Then the adjoint of $\id_{f^\ast \mathcal{C}_Y}$, i.e.\ the natural map $\mathcal{C}_Y \to f_\ast f^\ast \mathcal{C}_Y$ is a morphism of Cartier algebras.
\end{Le}
\begin{proof}
The map $\mathcal{C}_Y \to f_\ast f^\ast \mathcal{C}_Y$ is clearly $\mathcal{O}_Y$-linear. We still need to check that it is multiplicative. This in turn is local so that we may reduce to the case where $Y= \Spec R$. Then the adjoint of the identity is given by the map $\mathcal{C}_R \to \Gamma(X,f^*\mathcal{C}_R),\, \kappa \mapsto \kappa \otimes 1$ which is clearly multiplicative.
\end{proof}

\begin{Def}
If $f: X \to Y$ is a morphism of noetherian schemes, $\mathcal{C}_Y, \mathcal{C}_X$ are Cartier algebras with $f^\ast \mathcal{C}_Y \cong \mathcal{C}_X$ and $M$ is a $\mathcal{C}_X$-module then we endow $f_\ast M$ with a $\mathcal{C}_Y$-structure via the map of Lemma \ref{PushforwardCartierstructure} above.
\end{Def}

\begin{Prop}
\label{FPurePushforward}
Let $f: X \to Y$ be an affine morphism of noetherian schemes and let $M$ be a Cartier module on $X$. Then $f_\ast {\mathcal{C}_{X+}} M = \mathcal{C}_{Y+} f_\ast M$. In particular, $M$ is $F$-pure if and only if $f_\ast M$ is $F$-pure.
\end{Prop}
\begin{proof}
This is local so that we may assume that $X = \Spec S$ and $Y = \Spec R$. %we're asking for an inclusion of modules to be an equality.
The inclusion from right to left is clear. The other inclusion follows from $(\kappa \otimes s)M \subseteq (\kappa \otimes 1) M$.
\end{proof}

\begin{Bsp}
It is easy to see that for a proper morphism $f_\ast M$ is $F$-pure if $M$ is $F$-pure. The converse is not true. Consider a supersingular elliptic curve $f: X \to \Spec k$, where $k$ is a perfect field. Then $\omega_X$ endowed with a Cartier structure via the Cartier operator $\kappa$ is $F$-pure.

However, $f_\ast \omega_X = H^0(X, \omega_X) \cong k$ with the Cartier structure induced from $F_\ast f_\ast \omega \cong f_\ast F_\ast \omega \xrightarrow{f_\ast \kappa} f_\ast \omega_X$ is zero. This is due to the fact that the action of the Cartier operator on global sections corresponds via Grothendieck-Serre duality to the Frobenius action $H^1(X, \mathcal{O}_X) \to H^1(X, F_\ast \mathcal{O}_X)$ which, by our assumption, is zero.

In fact, for any smooth projective variety $X$ of dimension $d$ whose $p$-rank is not maximal one has that the Frobenius action on $H^d(X, \mathcal{O}_X) \to H^d(X, F_\ast \mathcal{O}_X)$ is not injective. This implies that the corresponding Cartier action on $F_\ast H^0(X,  \omega_X) \to H^0(X, \omega_X)$ is not surjective.
\end{Bsp}

\begin{Prop}
\label{PushforwardPreservesNilpotence}
Let $f: X \to Y$ be a morphism of noetherian schemes and let $M$ be a Cartier module on $X$. If $M$ is nilpotent then so is $f_\ast M$. If $f$ is affine the converse holds. Moreover, $f_\ast$ induces a functor on crystals.
\end{Prop}
\begin{proof}
Assume that $M$ is nilpotent. If $U \subseteq Y$ is open then $f_\ast M(U) = M(f^{-1}(U))$. The action of $\mathcal{C}_Y$ is induced by the natural map $\mathcal{C}_Y(U) \to \mathcal{C}_X(f^{-1}(U))$. By assumption $\mathcal{C}_{X+}^e(f^{-1}(U))$ acts as zero on $M(f^{-1}(U))$ for all sufficiently large powers $e$. We conclude that $\mathcal{C}_{Y+}^e(U)$ also acts as zero. Hence, if the open sets $U$ run through a covering of $Y$ we obtain some $e$ so that $\mathcal{C}_{Y+}^e$ annihilates $f_\ast M$. If $f$ is affine the converse follows from Proposition \ref{FPurePushforward}. Since $f_\ast$ preserves nilpotence it induces a functor on crystals.
\end{proof}

\begin{Bsp}
Consider $k = \mathbb{F}_p$, $f: \mathbb{P}^n_k \to \Spec k$ and $\omega = \mathcal{O}_{\mathbb{P}^n_k}(-n -1)$ with Cartier algebra generated by the Cartier operator. Then $f_\ast \omega = H^0(\mathbb{P}^n_k, \omega) = 0$ so that $f_\ast \omega$ is nilpotent, while clearly $\omega$ is not nilpotent.
\end{Bsp}

\section{Test modules and twisted inverse image}
\label{SectionTestmodulesPullbacks}

In this section we show that $F$-regularity (and hence the test module) is preserved by \'etale, and more generally, smooth twisted inverse images. Moreover, we will see that one has an (in general proper) inclusion $\tau(f^!M) \subseteq f^! \tau(M)$ in the finite case.

This is a local issue so that we may reduce the smooth case to the \'etale case and the situation $\mathbb{A}^n_R \to \Spec R$. The latter will be handled explicitly.
The idea to prove the \'etale case is as follows. Test element theory allows us to reduce to the situation where $f: \Spec S \to \Spec R$ is finite \'etale with $R$ a normal integral domain. With a little more work we may further assume that the extension of function fields is Galois, where we replace $\Spec S$ by any connected component. Then the claim follows from a Galois invariance argument.
In particular, this argument also removes the assumption that the Cartier algebra is generically principal that was necessary for the approach of \cite[Theorem 6.15]{staeblertestmodulnvilftrierung} to work.

Starting with this section we adopt the following convention: If $\Spec S \to \Spec R$ is a morphism and $\eta$ a prime in $R$ then the notation $S_\eta$ is short-hand notation for $S[ (R \setminus \eta)^{-1}]$.

\begin{Le}
\label{CartierOperationEtalePullbackCommutes}
If $M$ is a $\mathcal{C}_Y$-module and $f: X \to Y$ essentially \'etale or smooth then $\mathcal{C}_{X+} f^! M = f^! \mathcal{C}_{Y+} M$, where $\mathcal{C}_X=f^*\mathcal{C}_Y$.
\end{Le}
\begin{proof}
This is local on $X$ and hence on $Y$ so that we may assume that $f: \Spec S \to \Spec R$. We start with the essentially \'etale case. The inclusion from right to left is clear. On the other hand, given $\kappa \otimes s$ homogeneous of degree $e$ and $m \otimes t$ we may write $m \otimes st = \sum_i m_i \otimes u_i^{p^e}$ and then $\kappa \otimes s \cdot m \otimes t = \sum_i \kappa(m_i) \otimes u_i \in f^! \mathcal{C}_+ M$.

For the smooth case we may factor $f$ as an \'etale map $\varphi: \Spec S \to \Spec \mathbb{A}^n_R$ followed by the structural map $\mathbb{A}^n_R \to \Spec R$. It suffices to treat the latter case and to restrict to $n =1$, i.e. $f$ is the natural map $\Spec R[x] \to \Spec R$. For the inclusion from right to left if $u = dx \otimes r\otimes \kappa(m) \in f^!\mathcal{C}_{X+}M$, where $\kappa$ is homogeneous of degree $e$, then $dx \otimes r^{p^e} x^{p^e -1} \otimes m$ is mapped to $u$ by $\kappa \otimes 1$. The other inclusion is clear.
\end{proof}

\begin{Le}
\label{EtaleShriekAssPrimes}
Let $M$ be a $\mathcal{C}_Y$-module and let $f: X \to Y$ be an essentially \'etale morphism. Then $\Ass f^!M = f^{-1} \Ass M$.
\end{Le}
\begin{proof}
This is local on $X$ and hence also $Y$ so that we may assume that $f: \Spec S \to \Spec R$ is a morphism of affine schemes.
We may assume that $M$ is $F$-pure by Lemma \ref{AssDescendsToCrystals}. By Proposition \ref{TauFormalProperties} we may further assume that $\Supp M = \Spec R$ and that $R$ is reduced (it follows that $S$ is reduced since $f$ is \'etale). By Lemma \ref{LocalCohomShriekUpToNilpotence} we may identify $\underline{H^0_\eta(M)}$ with $\underline{i^!(M)}$, where $i: \Spec R/\eta \to \Spec R$ is the natural closed immersion and $\eta \in \Ass M$. Furthermore $\underline{H^0_\eta(M_\eta)} = \underline{H^0_\eta(M)}_\eta$ is then given by $j^!\underline{ i^! M}$, where $j: \Spec (R/\eta)_\eta \to \Spec R/\eta$ is the inclusion of the generic point. Since $F$-purity is preserved by essentially \'etale morphisms (use Lemma \ref{CartierOperationEtalePullbackCommutes} above) we may argue similarly for $\nu \in \Ass f^! M$.
We have the following pullback diagram
\[\begin{xy}
   \xymatrix{\Spec (R/\eta)_\eta \ar[r]^j & \Spec R/\eta \ar[r]^i & \Spec R \\
F \ar[r]^{j'} \ar[u]^{f''} & \Spec S/\eta S \ar[r]^{i'} \ar[u]^{f'} & \Spec S \ar[u]^f\\
}
  \end{xy}
\]
where $F = \Spec (S/\eta S)[(R \setminus \eta)^{-1}]$ is simply the fiber of $f$ over $\eta$. Since $f$ is \'etale the fiber is finite. Hence, if $\nu$ is a point in the fiber, which by abuse of notation we identify with a point in $\Spec S$, then the inclusion $\alpha: (\Spec S/\nu)_{\nu} \to F$ is open hence \'etale. We conclude that $\beta = f'' \circ \alpha: \Spec (S/\eta S)_\nu \to \Spec (R/\eta)_\eta$ is \'etale and surjective.

Assume now that $\eta \in \Ass M$ and that $\nu$ is a point in the fiber of $f$ over $\eta$. By assumption $N = j^! i^!M$ is not nilpotent. If $\beta^! j^! i^! M = H^0_\nu(f^!M)_\nu$ is nilpotent then using Lemma \ref{CartierOperationEtalePullbackCommutes} and faithful flatness of $\beta$ we conclude that $j^! i^! M$ is also nilpotent. This is a contradiction.
%This means, by definition, $N \neq N_{\text{nil}}$ (see Lemma \ref{OverlineM}). Since $\beta$ is faithfully flat and \'etale we conclude that $\beta^!N \neq \beta^!(N_{\text{nil}})$. By \cite[Lemma 6.10]{staeblertestmodulnvilftrierung} we have $\beta^!(N_{\text{nil}}) = (\beta^!N)_{\text{nil}}$. Thus $\beta^! j^! i^! M = H^0_\nu(f^! M)_\nu$ is not nilpotent.

Conversely, if $\nu \in \Ass f^! M$, i.e.\ $\beta^! j^! i^! M $ is not nilpotent, then a fortiori $j^! i^! M$ is not nilpotent.
\end{proof}

The following lemmata are well-known. However, we were unable to track down a precise reference.

\begin{Le}
\label{GInvariantsNiceEtaleSituation}
Let $f: \Spec S \to \Spec R$ be a finite \'etale morphism of normal integral schemes such that the extension of function fields is Galois. Given a coherent $R$-module $M$ there is a non-zero $c \in R$ such that $(M \otimes_R S)_c^G = M_c$, where $M \otimes_R S$ carries the $G$-action $m \otimes s \mapsto m \otimes g(s)$.
\end{Le}
\begin{proof}
First of all note that, due to normality of $S$, we have $g(s) \in S$.
The inclusion $M \subseteq (M \otimes_R S)^G$ is clear. By generic freeness we find $c$ such that $M_c$ and $(M \otimes_R S)_c$ are free. Choose a basis $m_1, \ldots, m_n$ for $M$ and assume that $m = \sum_i m_i \otimes s_i \in (M \otimes_R S)^G$. Then for any $g \in G$ we have $g m = m$ so that $\sum_i m_i \otimes g(s_i) - s_i  = \sum_i (g(s_i) - s_i) \cdot m_i \otimes 1= 0$. Since the $m_i \otimes 1$ form a basis for $M \otimes S$ we conclude that $g(s_i) = s_i$ for all $g \in G$. Since $Q(S)^G = Q(R)$ and $R$ is normal we obtain that $s_i \in R$ as desired.
\end{proof}

\begin{Le}
\label{GInvariantsGenericEquality}
Let $f: \Spec S \to \Spec R$ be a finite \'etale morphism of normal integral schemes such that the extension of function fields is Galois. If $N \subseteq f^! M$ is an $S$-submodule, where $M$ is free, and $N$ and $f^! M$ agree generically, then $N^G = N \cap (f^!M)^G$ and $M$ agree generically.
\end{Le}
\begin{proof}
Fix a basis $m_1, \ldots, m_n$ of $M$. Since $f^! M/N$ is generically zero we find $0 \neq c \in S$ such that $c (m_i \otimes 1) \in N$ for all $i=1, \ldots, n$. If $\mathcal{N}: S \to R$ denotes the restriction of the norm then the $\mathcal{N}(c) (m_i \otimes 1)$ are contained in $N^G$ and generically span all of $M$.
\end{proof} %if rk N < rk f^\ast M then it may happen that the rank of $N$ drops further.

In order to apply test element theory we need to assume that our base is essentially of finite type over an $F$-finite field in the next result.

\begin{Theo}
\label{EtalePullbackFRegularity}
Let $f: X \to Y$ be an essentially \'etale morphism with $Y$ essentially of finite type over an $F$-finite field and $M$ a $\mathcal{C}_Y$-module. If $M$ is $F$-regular then $f^!M$ is an $F$-regular $\mathcal{C}_X=f^*\mathcal{C}_Y$-module. If $f$ is surjective the converse holds.
\end{Theo}
\begin{proof}
$F$-regularity is local so that we may assume that $Y = \Spec R $ and $X= \Spec S$ are both affine.%Spelled out: Let $U_i$ be an open cover of $Y$. Then $M$ is $F$-regular iff M\vert_{U_i} is $F$-regular. In Turn $f^!M$ is $F$-regular if and only if $f^! M_{U_i}$ is $F$-regular iff $f^! M_{U_i}$ further restricted to some affine covering is $F$-regular.

By Lemma \ref{CartierOperationEtalePullbackCommutes} we may assume that both $M$ and $f^!M$ are $F$-pure. In particular, by making a base change we may assume that $R$ is reduced and $\Supp M = \Spec R$.

Assume that $f$ is surjective and that $M$ is not $F$-regular. Hence, $M$ admits a proper submodule $N$ for which the inclusion $H^0_\eta(N) \subseteq H^0_\eta(M)$ is a nil-isomorphism for all associated primes $\eta$ of $\Ass M$. By faithful flatness $f^!N$ is a proper submodule of $f^!M$ and $f^! H^0_\eta(M_\eta) = H^0_{\eta S}(f^! M_{\eta})$ (\cite[Lemma 4.3.1]{brodmannsharp}). Let now $\nu$ be an associated prime of $\Ass f^! M$ which by Lemma \ref{EtaleShriekAssPrimes} lies above some associated prime $\eta$ of $M$. From the proof of Lemma \ref{EtaleShriekAssPrimes} we deduce that \[\underline{H^0_\nu(f^!N)}_\nu = \underline{H^0_{\eta S}(f^! N)}_\nu \text{ and } \underline{H^0_{\eta S}(f^! M)}_\nu = \underline{H^0_\nu(f^!M)}_\nu.\] Since the inclusion $\underline{H^0_{\eta S}(f^! N)}_\nu \subseteq \underline{H^0_{\eta S}(f^! M)}_\nu$ is a nil-isomorphism we have that the inclusion $\underline{H^0_\nu(f^!N)_\nu} \subseteq \underline{H^0_\nu(f^!M)_\nu}$ is also a nil-isomorphism. But then also $H^0_\nu(f^!N)_\nu \subseteq H^0_\nu(f^!M)_\nu$ is a nil-isomorphism as desired.

For the other direction let us denote the associated primes of $M$ by $\eta_1, \ldots, \eta_n$. We claim that it is sufficient to find a sequence of test elements $c_1, \ldots, c_n \in R$ for $M$, where each $c_i$ is chosen in such a way that $H^0_{\eta_i}(M)_{c_i}$ has only one associated prime, namely $\eta_i$, such that $f^! (\underline{H^0_{\eta_i}(M)})_{c_i} = \underline{H^0_{\eta_i S}(f^!M)}_{c_i}$ is $F$-regular. Assume that this holds and fix an associated prime $\eta$ of $\Spec R$ and $c \in R$ as above. We note that the associated primes of $f^! H^0_{\eta_i}(M)_{c_i}$ are the elements in $f^{-1}(\eta_i)$ which we will denote by $\nu_1, \ldots, \nu_{m_i}$. We are therefore in the following situation:
\[\begin{xy} \xymatrix{\Spec (S/\eta S)_c \ar[d]^{f'} \ar[r]^{i'} & \Spec S_c \ar[d]^f\\
   \Spec (R/\eta)_c \ar[r]^i & \Spec R_c}
  \end{xy}
\]
Choosing a different $c$ we may assume that $(R/\eta)_c$ is normal. Since $f'$ is \'etale and $(\Spec R/\eta)_c$ is normal and integral $\Spec (S/\eta S)_c$ splits as a direct sum, where the irreducible (equivalently connected) components are given by $(\Spec S/\nu_j)_c$, for $j=1, \ldots, m$. Since all associated primes of $f^! H^0_\eta(M)_c$ are minimal we may invoke \cite[Lemma 6.13]{staeblertestmodulnvilftrierung} to see that the $\underline{H^0_{\nu_j}(f^!M)}_c$ are $F$-regular for $j = 1, \ldots, m$. Hence, the $c_i$ are test elements for the $\nu_{1}, \ldots, \nu_{m_i}$ and the $c_i$ (with appropriate repetitions) form a sequence of test elements for $f^! M$.

Observe that we have an inclusion $\sum_{j=1}^m \mathcal{C}_+ c \underline{H^0_{\nu_j}(f^!M)} \subseteq \underline{H^0_{\eta S}(f^! M)}$ which becomes an equality when localizing at $c$. % $H^0_\nu$ cuts out a connected component and their sum yields everything)
Hence, by Lemma \ref{LocalizatonAtcCoincidesInclusionAfterMultbyc} we get \[c \underline{H^0_{\eta S}(f^!M)} \subseteq \sum_{j=1}^m \mathcal{C}_+ c \underline{H^0_{\nu_j}(f^!M)}. \]
Using Theorem \ref{TestElementExistence} we conclude that $\tau(f^!M)$ must contain \[\sum_{i} \mathcal{C}_+ f^!(c_i \underline{H^0_{\eta_{i}}(M)}) = \sum_i f^!(\mathcal{C}_+ c_i \underline{H^0_{\eta_i}(M)}) = f^!\tau(M) = f^! M\] as claimed.

In order to finish the proof we argue that we may find test elements $c_1, \ldots, c_n \in R$ for $M$ with the required properties. Since $\Ass H^0_{\eta_i}(M)$ consists of associated primes containing $\eta_i$ we find $c$ such that the only associated prime of $H^0_{\eta_i}(M_{c_i})$ is $\eta_i$. Let us now fix $\eta_i$ and omit the index. If $i: \Spec R/\eta \to \Spec R$ denotes the closed immersion then Lemma \ref{CartierOperationEtalePullbackCommutes} implies that $f'^! \underline{i^!M} = \underline{i'^! f^! M}$, where $f': \Spec S/\eta S \to \Spec R/\eta R$ is the base change of $f$ and $i'$ the base change of $i$. Due to Remark \ref{ShriekLocCohomAbuseOfNotation} we may identify $\underline{i^! M}$ with $\underline{H^0_\eta(M)}$.

Replacing $c$ by $cd$ for suitable $d$ we may further assume that $\Spec (R/\eta)_c$ is regular and integral. In particular, $(S/\eta S)_c$ is a product of regular domains. By further restriction we may assume that $f: \Spec (S/\eta S)_c \to \Spec (R/\eta)_c$ is finite (\cite[Exercise III.3.7]{hartshornealgebraic}). Since we only have minimal associated primes we may moreover assume that $\Spec (S/\eta S)$ is integral by \cite[Lemma 6.13]{staeblertestmodulnvilftrierung}. Finally, by the first direction of the theorem we may assume that the extension of function fields is Galois.

Assume now that $f^! \underline{i^!M}_c$ is not $F$-regular. This means that we find a proper submodule $N \subseteq f^! \underline{i^!M}_c$ that generically coincides with $f^! \underline{i^! M}_c$. Note that the $G$-invariants $N^G$ are a $\mathcal{C}$-module. By Lemma \ref{GInvariantsNiceEtaleSituation} we obtain an inclusion $N^G \subseteq \underline{i^!M}_c$ which still induces a generic equality (Lemma \ref{GInvariantsGenericEquality}), where we once again change $c$. Since $\underline{i^!M}_c$ is $F$-regular this inclusion has to be an equality. But this contradicts the fact that $f^! (N^G) \subseteq N \subsetneq f^!M$. %Actually we don't even need that $N^G$ is a $\mathcal{C}-module. We could simply work with $\mathcaL{C} N^G$ since $f^!$ of that is still contained in $N$
\end{proof}

\begin{Ko}
\label{TauCommutesEtalePullback}
Let $f: X \to Y$ be an essentially \'etale morphism with $Y$ essentially of finite type over an $F$-finite field and $M$ a $\mathcal{C}_Y$-module. Then $f^! \tau(M) = \tau(f^!M)$.
\end{Ko}
\begin{proof}
We may replace $M$ by $\underline{M}$. By exactness of $f^!$ we have an inclusion $f^! \tau(M) \subseteq f^! M$. Since for any associated prime $\eta$ of $M$ we have nil-isomorphisms \[H^0_{\eta S}(f^! \tau(M))_{\eta} \subseteq H^0_{\eta S}(f^!M)_\eta\] we a fortiori have nil-isomorphisms \[H^0_{\nu}(f^! \tau(M))_\nu \subseteq H^0_\nu(f^! M)_\nu\] for any associated prime $\nu$ of $f^!M$.
Now $\tau(f^! M)$ is minimal with this property so that $\tau (f^!M) \subseteq f^! \tau(M)$. By Theorem \ref{EtalePullbackFRegularity} the module $f^! \tau(M)$ is $F$-regular which shows that equality holds.
\end{proof}

We now turn to the smooth case:

\begin{Le}
\label{LocalCohomSmoothShriekFPureCompatible}
Let $g: \mathbb{A}^n_R \to \Spec R$ be the natural morphism. If $M$ is a Cartier module and $\eta$ a prime in $R$ then $g^! \underline{H^0_\eta(M)} = \underline{H^0_{\eta R[x]}(g^!M)}$.
\end{Le}
\begin{proof}
It suffices to prove the case $n =1$. Since $\omega_{R[x]/R}$ is free of rank one and $g^! A = g^\ast A \otimes_{R[x]} \omega_{R[x]/R} = A \otimes_R \omega_{R[x]/R}$ this follows from flatness of $g$ and the fact that $\mathcal{C}'_+ g^! M = g^! \mathcal{C}_+ M$ for a Cartier module $M$ (Lemma \ref{CartierOperationEtalePullbackCommutes}), where $\mathcal{C}' = g^\ast \mathcal{C}$.
\end{proof}

\begin{Le}
\label{AssPrimesSmoothPullback}
Let $g: \mathbb{A}^n_R \to \Spec R$ be the natural morphism. Then the map \[s: \Spec R \to \mathbb{A}^n_R, \quad \eta \mapsto \eta R[x_1, \ldots, x_n]\] induces a bijection between $\Ass M$ and $\Ass g^! M$.
\end{Le}
\begin{proof}
Again we only need to prove this for $n=1$. By \cite[Theorem 23.2 (ii)]{matsumuracommutativeringtheory} (note that the $\Ass_A$ should read $\Ass_B$) the associated primes of the $S$-module $M \otimes_R R[x]$ are precisely the $\eta R[x]$ for which $\eta$ is an associated prime of the $R$-module $M$.

Assume that $\eta$ is associated to $M$ and note that we may identify $g^!$ with $g^\ast$ since the relative dualizing sheaf $\omega = R[x] dx$ is trivial. We have $g^!_\eta H^0_\eta(M_\eta) = H^0_{\eta S}(g^!_\eta M_\eta) = H^0_{\eta S}((g^! M)_{\eta})$. Since $g^!_\eta$ and localization preserve $F$-purity by Lemma \ref{CartierOperationEtalePullbackCommutes} we conclude that $\eta R[x]$ is an associated prime of $g^!M$.
\end{proof}

\begin{Theo}
\label{SmoothPullbackFRegularity}
Let $f: X \to Y$ be a smooth morphism, where $Y$ is essentially of finite type over an $F$-finite field. If $M$ is an $F$-regular Cartier module then $f^! M$ is also $F$-regular.
\end{Theo}
\begin{proof}
By Proposition \ref{TauFormalProperties} (b) the issue whether $f^! M$ is $F$-regular is local on $X$. Hence, we reduce to the situation $f: \Spec S \to \Spec R$ and may factor $f$ as $g \circ \varphi$, where $\varphi$ is \'etale and $g: \mathbb{A}^n_R \to \Spec R$ is the natural map. The \'etale case is Theorem \ref{EtalePullbackFRegularity}, hence we only need to show that $g^!M$ is $F$-regular if $M$ is $F$-regular. As before it suffices to deal with the case $n=1$.

Denote the associated primes of $M$ by $\eta_1, \ldots, \eta_n$. By Theorem \ref{TauImpliesTestElements} above we find test elements $c_1, \ldots, c_n$ for $M$ such that the only associated prime of $H^0_{\eta_i}(M)_{c_i}$ is $\eta_i$. By Lemma \ref{AssPrimesSmoothPullback} the associated primes of $g^! M$ are given by $\eta_1 R[x], \cdots, \eta_n R[x]$. Since the associated primes of $\underline{H^0_{\eta_i}(M_{c_i})}$ are all minimal we conclude using Proposition \ref{OldNewTestModuleRelation}, Lemma \ref{LocalCohomSmoothShriekFPureCompatible} and \cite[Lemma 3.6]{staeblerunitftestmoduln} that $\underline{H^0_{\eta R[x]}(g^! M)}_{c_i}$ is $F$-regular. In particular, the $c_i$ are test elements for $\eta_i R[x]$. Hence, \begin{align*}\tau(g^!M) &= \sum_i \mathcal{C}'_+ c_i \underline{H^0_{\eta_i R[x]}(g^! M)} \\&= \sum_i \mathcal{C}'_+ c_i \underline{H^0_{\eta_i}(M)} \otimes_R R[x] dx = \tau(M) \otimes_R R[x]dx = g^!M.\end{align*}
\end{proof}

\begin{Ko}\label{TauCommutesSmoothPullback}
Let $f: X \to Y$ be a smooth morphism, where $Y$ is essentially of finite type over an $F$-finite field. Then for any coherent $\mathcal{C}$-module $M$ one has $f^! \tau(M) = \tau(f^!M)$.
\end{Ko}
\begin{proof}
Since this is local and using Corollary \ref{TauCommutesEtalePullback} we may immediately reduce to the situation that $f: \mathbb{A}^1_R \to \Spec R$. The argument proceeds now similarly to Corollary \ref{TauCommutesEtalePullback} using Lemma \ref{AssPrimesSmoothPullback}.
\end{proof}

\begin{Ko}
\label{GrCommutesSmoothPullback}
Let $f: X \to Y$ be a smooth morphism, where $Y$ is essentially of finite type over an $F$-finite field, $\mathfrak{a}$ an ideal sheaf and $t \in \mathbb{Q}$. Then for any coherent $\mathcal{C}$-module $M$ one has $f^! (\tau(M, \mathfrak{a}^{t-\eps})/\tau(M, \mathfrak{a}^t))= \tau(f^!M, \mathfrak{a}^{t-\eps} \mathcal{O}_X)/\tau(f^!M, \mathfrak{a}^t \mathcal{O}_X)$.
\end{Ko}
\begin{proof}
Immediate from Corollary \ref{TauCommutesSmoothPullback} and right-continuity (Proposition \ref{Rightcontinuity}) since $f^!$ is exact.
\end{proof}

\begin{Le}
\label{FiniteShriekAssPrimes}
Let $f: X \to Y$ be a finite dominant morphism with $Y$ $F$-finite and $M$ a $\mathcal{C}_Y$-module. Then $\Ass f^! M = f^{-1} \Ass M$.
\end{Le}
\begin{proof}
This is local and so reduces to a situation $f:\Spec S \to \Spec R$.
The proof is similar to the one of Lemma \ref{EtaleShriekAssPrimes}. Given a prime $\eta \in \Spec R$ and $\nu$ a prime in the fiber of $\eta$ we consider the pullback square

\[
\begin{xy}
\xymatrix{
\Spec (S/\eta S)_{\eta} \ar[d]^{f'}\ar[r]^{\alpha'} & \Spec S \ar[d]^f\\
\Spec (R/\eta)_\eta \ar[r]^\alpha & \Spec R }
\end{xy}
\]
and the closed immersion $\beta: \Spec (S/\nu)_{\nu} \to \Spec (S/\eta S)_{\eta}$. Due to Lemma \ref{LocalCohomShriekUpToNilpotence} the prime $\eta$ is associated to $M$ if and only if $\alpha^! M$ is not nilpotent. Similarly, $\nu$ is associated to $f^!M$ if and only if $\beta^! f'^! \alpha^!M$ is not nilpotent. In particular, Corollary \ref{Nonnilpotencepreservedbyshriek} implies that if $\nu$ is associated to $f^! M$ then $\nu \in f^{-1} \Ass M$.

Let us denote the composition $f' \circ \beta$ by $\gamma$ and note that $\gamma$ corresponds to a finite field extension $K \subseteq L$. Hence, we have to show that if $M$ a non-nilpotent $K$-vector space then $\gamma^! M = \Hom_K(L, M)$ is not nilpotent as an $L$-vector space. Recall that if $\kappa: F_\ast^e M \to M$ corresponds to the action of a homogeneous element of $(\mathcal{C}_+)^h$ then the corresponding action on $\gamma^! M$ is given by the formula $\varphi \mapsto \kappa \varphi F^e$. If $M$ is not nilpotent we find one such $\kappa$ and $m \in M$ such that $\kappa(m)$ is non-zero. Since $1$ is part of a $K$-basis of $L$ we find $\varphi \in \gamma^! M$ such that $\varphi(1) = m$. We then obtain $\kappa \varphi(F^e(1)) = \kappa(m) \neq 0$ which shows that $\gamma^! M$ is not nilpotent.
\end{proof}

\begin{Prop}
\label{FiniteShriekTauInclusion}
Let $f: X \to Y$ be a finite dominant map and $M$ a $\mathcal{C}_Y$-module. Then $\tau(f^!M) \subseteq f^! \tau(M)$.
\end{Prop}
\begin{proof}
This is local so that we may assume that $X = \Spec S$ and $Y = \Spec R$.
We will show that for $\nu \in \Ass f^!M$ the inclusion $H^0_\nu(\underline{f^! \tau(M)})_{\nu} \subseteq H^0_\nu(\underline{f^!M})_\nu$ is a nil-isomorphism. Since $\tau(f^!M)$ is minimal with this property we then obtain $\tau(f^!M) \subseteq \underline{f^! \tau(M)} \subseteq f^! \tau(M)$. Let us write $f(\nu) = \eta$. From the diagram in Lemma \ref{FiniteShriekAssPrimes} we obtain that $\underline{f'^! H^0_{\eta}(M)}_\eta = \underline{H^0_{\eta S}(f^! M)}_{\eta}$ which lives on the reduced fiber $F  = \Spec (S/\nu_1)_{\nu_1} \times \ldots \times (S/\nu_n)_{\nu_n}$, where each $(S/\nu_i)_{\nu_i}$ is a finite field extension of $(R/\eta)_\eta$ (with $\nu = \nu_1$ say). If $j: \Spec (S/\nu)_\nu \to F$ denotes the inclusion then we obtain that $j^!\underline{H^0_{\eta S}(f^! M)}_{\eta} = \underline{H^0_\nu(f^!M)}_\nu$. Consider now the inclusion $H^0_\eta(\tau(M))_\eta \subseteq H^0_\eta(M)_\eta$ which is a nil-isomorphism. We obtain that the inclusion $H^0_\nu(\underline{f^! \tau(M)})_{\nu} \subseteq H^0_\nu(\underline{f^!M})_\nu$ is still a nil-isomorphism since $f^!$, $j^!$ are left-exact, $\underline{f^! M} \subseteq f^!M$ and $\underline{f^! \tau(M)} \subseteq f^! \tau(M)$ are nil-isomorphisms and since $H^0_\nu$, localization and $f^!$ preserve nil-isomorphisms (Lemma \ref{H0FunctorOnCrys} and Corollary \ref{Nonnilpotencepreservedbyshriek}).
\end{proof}

Equality does not hold in general (cf.\ \cite[Example 3.13]{staeblerunitftestmoduln} and note that the Cartier module in the example only has minimal associated primes).

Let us write $Gr(M, \mathfrak{a}^t)$ for the quotient $\tau(M, \mathfrak{a}^{t-\eps})/\tau(M, \mathfrak{a}^t)$. With this notation we have:

\begin{Ko}
\label{FiniteShriekGrInclusion}
Let $f: \Spec S \to \Spec R$ be a finite dominant map and $M$ a $\mathcal{C}_R$-module and $\mathfrak{a}$ an ideal. Then $Gr(f^! M, \mathfrak{a}^t) \subseteq f^! Gr(M, \mathfrak{a}^t)$.
\end{Ko}
\begin{proof}
We have the following commutative diagram with exact rows, where the vertical arrows are injective (due to Proposition \ref{FiniteShriekTauInclusion}):
\[\begin{xy}\xymatrix{ 0 \ar[r] & \tau(f^!M, \mathfrak{a}^t) \ar[r] \ar[d]& \tau(f^!M, \mathfrak{a}^{t-\eps}) \ar[r] \ar[d]& Gr(f^!M, \mathfrak{a}^t) \ar[r] &0\\
 0 \ar[r] & f^! \tau(M, \mathfrak{a}^t) \ar[r] & f^! \tau(M, \mathfrak{a}^{t-\eps}) \ar[r] & f^! Gr(M, \mathfrak{a}^t)}
\end{xy}\]
One sees that this induces the desired inclusion.
\end{proof}

\section{A relative notion of gauge boundedness}
\label{SectionRelativeGaugebounds}

The goal of this section is to prove that, under suitable conditions on $\mathcal{C}_Y$, for $f: X \to Y$ a morphism of finite type and $M$ a coherent $\mathcal{C}_X$-module on $X$ one finds a \emph{coherent} $\mathcal{C}_Y$-submodule $N \subseteq f_\ast M$ for which the inclusion is a \emph{local nil-isomorphism} (see \ref{DefLNilpotence} below). This enables us to make sense of $\tau(f_\ast M)$.

The idea of the proof is a modified version of the argument used in \cite[Theorem 3.2.14]{blickleboecklecartiercrystals}, where this is proven in the situation of a single Cartier linear operator. One needs to incorporate techniques used in \cite{blicklep-etestideale} on gauge boundedness to make this work in greater generality.

\begin{Def}
Let $S$ be a finite type $R$-algebra (for some ring $R$) and fix a presentation $S = R[x_1, \ldots, x_n]/I$. For a multi-index $i = (i_1, \ldots, i_n)$ we denote by $\abs{i}$ its maximum norm. This induces an increasing filtration $F_d$ on $R[x_1, \ldots, x_n]$, where $F_d$ is the $R$-module generated by monomials $x_1^{i_1} \cdots x_n^{i_n}$ with $\abs{i} \leq d$ and $F_{-\infty} = 0$. This in turn induces an increasing filtration $S_d$ on $S$ given by $S_d = F_d \cdot 1$ for $d \geq 0$ and $S_{-\infty} = 0$.

Let now $M$ be a coherent $S$-module and fix generators $m_1, \ldots, m_r$. Then we can define an increasing filtration $M_d = S_d \cdot \langle m_1, \ldots, m_d\rangle$, $M_{-\infty} = 0$ on $M$. Given $m \in M$ we write $\delta(m) = d$, where $d$ is defined by the property $m \in M_d \setminus M_{d-1}$. We call $\delta$ a \emph{gauge} for $M$.
\end{Def}

Note that a gauge depends both on the presentation of $S$ and on the choice of generators of $M$. One should think of a gauge as a substitute for the degree in a polynomial ring. As was already exploited by Blickle (for $R = k$ an $F$-finite field) in \cite{blicklep-etestideale} gauges are a measure for the contracting property of Cartier linear operators.

For the next definition, recall Lemma \ref{Cartiermodulestructureviahomomorphism} which said that a $\mathcal{C}_S$-module structure on $M$ is equivalent to a ring homomorphism $\mathcal{C}_S \to \bigoplus \Hom_S(F_\ast^e M, M)$.

\begin{Def}
Let $f: \Spec S \to \Spec R$ of finite type, $\mathcal{C}_R$ a Cartier algebra and $f^\ast \mathcal{C}_R = \mathcal{C}_S$. Given a coherent $\mathcal{C}_S$-module $M$ we say that $(M, \mathcal{C}_S)$ is \emph{gauge bounded} if for some gauge $\delta$ on $M$ there exists a subset $\{\kappa_{e_i}\ \, \vert \, i \in I\} \subseteq {\mathcal{C}_S}_+$ of homogeneous elements of degrees $e_i$ such that their images in $\bigoplus_e \Hom_S(F_\ast^e M, M)$ generate the image of ${\mathcal{C}_S}_+$ as a \emph{right $S$-module} and a constant $K$ such that for all $i$ we have $\delta(\kappa_{e_i}(m)) \leq \frac{\delta(m)}{p^{e_i}} + K$.
We say that $\mathcal{C}_S$ is \emph{gauge bounded} if for all coherent $\mathcal{C}_S$-modules $M$ the pairs $(M, \mathcal{C}_S)$ are gauge bounded.
\end{Def}

As noted earlier any Cartier algebra $\mathcal{C}$ finitely generated over a ring $R$ essentially of finite type over an $F$-finite field is gauge bounded. If $\mathcal{C}$ is a gauge bounded Cartier algebra over $\Spec R$ then any algebra of the form $\mathcal{C}^{\mathfrak{a}^t}$ for an ideal $\mathfrak{a}$ and a rational number $t \geq 0$ is also gauge bounded (see \cite[Propositions 4.9, 4.15]{blicklep-etestideale}).

\begin{Def}
\label{DefLNilpotence}
We say that a quasi-coherent $\mathcal{C}$-module $M$ is \emph{locally nilpotent} if it admits a filtration $\bigcup_{e \in \mathbb{N}} M_e = M$ of nilpotent $\mathcal{C}$-submodules $M_e$. If $\varphi: M \to N$ is a morphism of quasi-coherent $\mathcal{C}$-modules then we say that $\varphi$ is a \emph{local nil-isomorphism} if both $\ker \varphi$ and $\coker \varphi$ are locally nilpotent.
\end{Def}

In the following results one obtains a trivial sharpening if one replaces the condition on the algebra $\mathcal{C}$ with the corresponding condition on the image of the algebra under the ring homomorphism $\mathcal{C} \to \bigoplus_e \Hom(F_\ast^e M, M)$.

\begin{Le}
\label{GaugeboundPushforwardNiliso}
Let $f: \Spec S \to \Spec R$ be of finite type, $\mathcal{C}$ a Cartier algebra on $\Spec R$ with homogeneous right $R$-module generators $\{\kappa_{i} \, \vert \, i \in I\}$ and $\mathcal{C}' = f^\ast \mathcal{C}$. If $M$ is a  coherent $\mathcal{C}'$-module, gauge bounded with respect to the $\kappa_{i} \otimes 1$, then $f_\ast M$ is locally nil-isomorphic to a coherent $\mathcal{C}$-module.
\end{Le}
\begin{proof}
By assumption we find $K$ such that for all $m \in M$ we have $\delta(\kappa_i \otimes 1(m)) \leq \frac{\delta(m)}{p^{e_i}} + K$, where $e_i$ is the degree of $\kappa_i$. Denoting the elements of gauge at most $e$ by $M_{\leq e}$ we claim that the inclusion $\underline{\mathcal{C} M_{\leq K}} \subseteq f_\ast M$ is the desired local nil-isomorphism.

Any homogeneous element $\kappa$ of degree $e$ in $\mathcal{C}$ may be written as a sum $\sum_i \kappa_i r_i$, where $r_i \in R$. Since $\delta(\sum_i m_i) \leq \max\{ \delta(m_i) \, \vert \, i\}$ and $R$ is contained in gauge zero any such $\kappa$ satisfies the inequality
\[\delta(\kappa(m)) \leq \frac{\delta(m)}{p^e} + K.\]
Let now $A \geq K +1$ then \[ \delta(\mathcal{C}_+^a M_A) \leq \frac{A}{p^a} + K.\] We conclude that the elements of $\mathcal{C}_+^a(M_A)$ have gauge $\leq K$ for all $a \gg 0$ since $\delta$ is integer valued. By definition $\underline{\mathcal{C} M_{\leq K}} = \mathcal{C}_+^b M_{\leq K} + \mathcal{C}_+^{b+1} M_{\leq K}$ for all $b \gg 0$. Hence, $\mathcal{C}_+^{2c} M_{A} = \mathcal{C}_+^c (\mathcal{C}_+^c M_A) \subseteq \mathcal{C}_+^c M_{\leq K} \subseteq \underline{\mathcal{C} M_{\leq K}}$ for all $c \geq \max\{a, b\}$. Note that $\underline{\mathcal{C} M_{\leq K}}$ is clearly a $\mathcal{C}$-module. Moreover, $\mathcal{C} M_{\leq K}$ consists of elements of gauge at most $\frac{K}{p} + K$. Hence, by the above inequality $\underline{\mathcal{C} M_{\leq K}}$ consists of elements of gauge at most $K$ and is thus coherent.
\end{proof}

\begin{Bem}
Unfortunately we do not know how to get rid of the assumption that one has a gauge bound on generators of the form $\kappa \otimes 1$. The issue is that if we start with some arbitrary generating system of $\mathcal{C}'$, then we do not get any control over the behavior of the algebra $\mathcal{C}$ on $f_\ast M$.
\end{Bem}

We will now show that the hypothesis of Lemma \ref{GaugeboundPushforwardNiliso} is satisfied if $\mathcal{C}$ is finitely generated or if $\mathcal{C} = \mathcal{D}^{\mathfrak{a}^t}$ for some finitely generated Cartier algebra $\mathcal{D}$, $\mathfrak{a}$ an ideal and $t$ a positive rational number.% This is similar to Blickle's proof of gauge boundedness for finitely generated $R$-Cartier algebras for $R$ of finite type over an $F$-finite field. In replacing the $F$-finite field by an $F$-finite ring $R$ we no longer have that the Frobenius is free. This is circumvented by appealing to a result of Gabber (\cite[Remark 13.6]{gabbertstructures}) which says that $R$ is a quotient of an $F$-finite ring $S$ on which the Frobenius is free (in particular, $S$ is regular).

\begin{Prop}
\label{GaugeBoundSingleMorphism}
Let $f: \Spec S \to \Spec R$ be of finite type, $M$ a coherent $S$-module endowed with a gauge $\delta$. Assume that $R$ is $F$-finite. If $\varphi: F_\ast^eM \to M$ is an $S$-linear map, then there is a constant $K$ such that for all $m \in M$ we have \[ \delta(\varphi(m)) \leq \frac{\delta(m)}{p^e} + \frac{K}{p^e -1}.\]
\end{Prop}
\begin{proof}
By \cite[Remark 13.4]{gabbertstructures} there is a closed embedding $\Spec R \to \Spec A$ such that $F: A \to F_\ast A$ is finite free. We may therefore replace $R$ with $A$ and assume that the Frobenius $F$ is free on $R$. Next, since $f$ is of finite type we may replace $S$ by a polynomial ring $R[x_1, \ldots, x_n]$. Since the monomials $x_1^{e_1} \cdots x_n^{e_n}$ with $e_i \leq q -1$ are a basis for $S$ over $R[x_1^q, \cdots, x_n^q]$ and since $F_\ast^e R$ admits a finite $R$-basis the rest of the proof works just as in \cite[Lemma 4.4, Proposition 4.5]{blicklep-etestideale}.
\end{proof}

\begin{Prop}
\label{FiniteGenerationGaugeBoundWithSpecialGenerators}
Let $f: \Spec S \to \Spec R$ be of finite type with $R$ $F$-finite, $\mathcal{C}$ a Cartier algebra finitely generated as an $R$-Cartier algebra. Then there are generators $\{\kappa_i \otimes 1 \, \vert \, i \in I\}$ such that for any coherent $\mathcal{C}'_S= f^\ast \mathcal{C}$-module $M$ the pair $(M, \mathcal{C}')$ is gauge bounded.
\end{Prop}
\begin{proof}
Fix algebra generators $\varphi_1, \ldots, \varphi_n$ of $\mathcal{C}$. Then the set $\{\psi \, \vert \, \psi$ finite product of $\varphi_i \otimes 1\}$ is a set of $S$-module generators for $\mathcal{C}_S$. It follows from Proposition \ref{GaugeBoundSingleMorphism} and an argument similar to the one in \cite[Corollary 4.6]{blicklep-etestideale} that $M$ is gauge bounded with respect to the $\psi$'s.
\end{proof}

\begin{Prop}
\label{FG+IdealGaugeBoundWithSpecialGenerators}
Let $f: \Spec S \to \Spec R$ be of finite type with $R$ $F$-finite, $\mathcal{C}$ a Cartier algebra finitely generated as an $R$-Cartier algebra, $\mathfrak{a} \subseteq R$ an ideal and $t \geq 0$ a rational number.
Then there are generators $\{ \kappa_i \otimes 1\, \vert\, i \in I\}$ of $\mathcal{C}_S= f^\ast(\mathcal{C}^{\mathfrak{a}^t})$ such that for any $\mathcal{C}_S$-module the pair $(M, \mathcal{C}_S)$ is gauge bounded.
\end{Prop}
\begin{proof}
By Proposition \ref{FiniteGenerationGaugeBoundWithSpecialGenerators} we find elements $\{\kappa_i \otimes 1\, \vert \, i \in I\}$ that generate $f^\ast \mathcal{C}$ such that $(M, f^\ast \mathcal{C})$ is gauge bounded for any $f^\ast\mathcal{C}$-module $M$. Since $\mathfrak{a} \subseteq R$ the ideal $\mathfrak{a}^{\lceil tp^e\rceil}$ is generated in gauge zero. Then the $\kappa_i a_{je_i} \otimes 1$ with $\kappa_i \otimes 1 \in \mathcal{C}_{e_i}$ and $a_{j e_i}$, $j = 1, \ldots, n = n(e_i)$, generators of $\mathfrak{a}^{\lceil t p^{e_i} \rceil}$ form a system of generators for $\mathcal{C}'$ which is gauge bounded.
\end{proof}

In order to state the main result of this section we need one more definition.

\begin{Def}
Let $X$ be a noetherian $F$-finite scheme and $\mathcal{C}_X$ a Cartier algebra on $X$ we say that $\mathcal{C}_X$ is \emph{finitely generated} if there exists a finite covering $U_i = \Spec A_i$ of open affine such that $\mathcal{C}_{U_i}$ is finitely generated as an $A_i$-Cartier algebra for all $i$.
\end{Def}

\begin{Theo}
\label{PushforwardCoherentCrystals}
Let $f: X \to Y$ a morphism of finite type between noetherian $F$-finite schemes. Let $\mathcal{C}_Y$ be a finitely generated Cartier algebra on $Y$, $\mathfrak{a} \subseteq \mathcal{O}_Y$ an ideal sheaf and $t \geq 0$ a rational number. If $M$ is a coherent $f^\ast \mathcal{C}_Y^{\mathfrak{a}^t}$-module, then there is a coherent $\mathcal{C}_Y^{\mathfrak{a}^t}$-module $N \subseteq f_\ast M$ for which the inclusion is a local nil-isomorphism.
\end{Theo}
\begin{proof}
Since $f_\ast$ may be defined by applying $f_\ast$ to a \v{C}ech resolution the argument is local on $X$. Hence, one only needs to consider affine morphisms $f:\Spec S \to \Spec R$ of finite type. This in turn follows from Lemma \ref{GaugeboundPushforwardNiliso} and Proposition \ref{FG+IdealGaugeBoundWithSpecialGenerators}
\end{proof}

\begin{Bem}
Note that if we only care about the coherence of a particular pushforward $f_\ast M$ then it is sufficient that the image of $\mathcal{C}_X$ in $\bigoplus_{e \geq 0} \mathcal{H}om(F_\ast^e M, M)$ locally admits a gauge bounded generating system of the form $\kappa_i \otimes 1$ as in Lemma \ref{GaugeboundPushforwardNiliso}. At present we are not aware of a naturally occurring example where this is not satisfied (cf.\ Remark \ref{CircumventGaugeBoundedness} below).
\end{Bem}

With this theorem we can now define the test module for a pushforward of finite type:

\begin{Def}
\label{DefTestmodulePushforward}
Let $f: X \to Y$ be a morphism of finite type between noetherian $F$-finite schemes. Let $\mathcal{C}_Y$ be a finitely generated Cartier algebra on $Y$, $\mathfrak{a} \subseteq \mathcal{O}_Y$ an ideal sheaf and $t \geq 0$ a rational number. Then for any coherent $f^\ast \mathcal{C}_Y^{\mathfrak{a}^t}$-module we define $\tau(f_\ast M, \mathcal{C}_Y^{\mathfrak{a}^t})$ as $\tau(N, \mathcal{C}_Y^{\mathfrak{a}^t})$ for any $N$ as in Theorem \ref{PushforwardCoherentCrystals}.
\end{Def}

This is independent of the choice of $N$ for if $N' \subseteq f_\ast M$ is another such submodule then $\underline{N} = \underline{N'}$ as $\mathcal{C}_Y$-submodules of $f_\ast M$. The proof is similar to the one of \cite[Lemma 2.4]{staeblerunitftestmoduln}. Since $\tau(N) = \tau(\underline{N})$ the claim follows.

Just as in \cite{blickleboecklecartiercrystals} one can develop a theory of pushforwards more generally in the category of crystals under the above assumptions on the Cartier algebra. The details are a relatively straightforward generalization of \cite{blickleboecklecartiercrystals}.

We conclude this section with an example that shows that the assertion of the theorem does not hold without a gauge boundedness assumption.

\begin{Bsp}
\begin{enumerate}[(a)]
 \item Let $\mathcal{C}$ be a principal Cartier algebra over $R$ with generator $\kappa$. Let $\mathcal{D}$ be a Cartier subalgebra over $R$. Then each graded component $\mathcal{D}_e$ is of the form $\kappa^e\mathfrak{a}_e$ for some ideal $\mathfrak{a}_e$ of $R$. The inclusion $\mathcal{D}_e \mathcal{D}_f \subseteq \mathcal{D}_{e+f}$ implies that $(\mathfrak{a}_e)_{e \in \mathbb{N}}$ is a $F$-graded system of ideals, i.e. $\mathfrak{a}_e^{[p^f]} \mathfrak{a}_f \subseteq \mathfrak{a}_{e+f}$, where the square brackets denote the image of $\mathfrak{a}_e$ under the $f$th iterate of the Frobenius.
\item If $R$ is a principal ideal domain of finite type over an $F$-finite field, then each $\mathfrak{a}_e = (a_e)$ for some element $a_e \in R$. A system of right $R$-module generators of $\mathcal{D}$ is therefore given by $\kappa^e a_e, e \in \mathbb{N}$. The $F$-graded system property implies that $a_{e+f} | a_e^{p^f}a_f$, and in particular $a_e | a_1^{1+p+p^2+\ldots+p^{e-1}} | a_1^{p^e}$.

If $M$ is any finitely generated $R$-module with gauge $\delta$ induced by a gauge $\delta$ on $R$, then
\[
\delta(\kappa^e(a_e m)) \leq \frac{\delta(a_em)}{p^e} + K \leq \frac{\delta(a_1^{p^e}) + \delta(m)}{p^e} + K \leq \frac{\delta(m)}{p^e} + \delta(a_1) + K,\]
where $K$ is some constant as in Proposition \ref{GaugeBoundSingleMorphism}. Hence every Cartier subalgebra of a principal $R$-Cartier algebra with $R$ a principal ideal domain is gauge bounded on every finitely generated $R$-module.
\item The following example shows that for higher dimensional $R$ the corresponding result is not true.
Let $R=k[x,y]$ with gauge $\delta$ given by grading of the variables. Consider the Cartier subalgebra $\mathcal{D}$ given by the graded system of monomial ideals $\mathfrak{a}_e = (x^2,xy^{ep^e})$. It is straightforward to check that this is a graded system (essentially since $x^2$ is in each of the ideals). Now the operators $\kappa^e x^2,\kappa^e xy^{ep^e}$ are generators of $\mathcal{D}_e$ as a right $R$-module. Letting $\kappa$ act on $R$ as the Cartier operator, we get
\[
(\kappa^e xy^{ep^e})(x^{p^e-2}y^{p^e-1}) = \kappa^e((y^e)^{p^e}(xy)^{p^e-1})=y^e \]
Hence $\delta((\kappa^e xy^{ep^e})(x^{p^e-2}y^{p^e-1}))=e$. But this is a contradiction to the gauge boundedness since $\delta(x^{p^e-2}y^{p^e-1})/p^e\leq 1$. This shows that the gauge boundedness cannot hold for the particular choice of right $R$-generators of the Cartier algebra. Assume now that $C_i$ are another set of generators in degree $e$. Then we can write $C_i = \kappa^e x^2 a_{i1} + \kappa^e x y^{ep^e} a_{i2}$ and since the $C_i$ are generators we obtain $\sum_i C_i r_i = \kappa^e x y^{ep^e}$ for some $r_i \in R$. By a degree argument this is only possible if for some $i$ the constant terms of $a_{i2}$ and of $r_i$ are non-zero. We conclude that for any such $i$ one has $\delta(C_i(x^{p^e-2}y^{p^e-1})) \geq \delta(y^e) = e$ which shows that this set of generators is also not gauge bounded.

Note that $R$ is $F$-pure. To see this it suffices to show that all monomials $x^ay^b$ are in the image of $\mathcal{C}_+ R$. Taking, for instance, $e = 2$ we have $\kappa^2 x^2 \cdot x^{p^2 - 3} y^{p^2 -1} x^{ap^2} y^{bp^2} = x^ay^b$.

\item A variation of the example discussed in (c) also shows that the pushforward of a coherent Cartier module along a morphism of finite type need not be coherent in general. Consider the morphism $f: \Spec \mathbb{F}_p[x,y] \to \mathbb{F}_p$ and let $\mathcal{C}$ be the free non-commutative algebra \[\mathbb{F}_p[\kappa_{i,e}, c_{a,b} \, \vert \, i =1,2; e \geq 1, a,b \geq 0]\] which we endow with an $\mathbb{N}$-grading by assigning $\kappa_{i,e}$ degree $e$ and $c_{a,b}$ degree $2$. Then we let $f^\ast \mathcal{C}$ act on $R = \mathbb{F}_p[x,y]$ by letting $\kappa_{1,e}$ act as $\kappa x^2$, $\kappa_{2,e}$ as $\kappa x y^{ep^e}$ and the $c_{a,b}$ act as $\kappa^2 x^{ap^2} y^{bp^2} (xy)^{p^2-1}$, where $\kappa$ is the Cartier operator. Note that the Cartier algebra $f^\ast \mathcal{C}$ acts on $R$ just as the one in (c) and hence is also not gauge bounded.

Since $c_{a,b}(1) = x^a y^b$ the Cartier submodule of $f_\ast R$ generated by $1$ is in fact $f_\ast R$. This shows that $\mathcal{C}_+^h \mathbb{F}_p$ contains monomials of arbitrary high degree for all $h \in \mathbb{N}$. If there was $V \subseteq f_\ast R$ finite dimensional and locally nil-isomorphic to $f_\ast R$ via the inclusion then in particular $\mathcal{C}_+^h \mathbb{F}_p$ would be contained in $V$ for all $h \gg 0$. This shows that $f_\ast R$ does not admit any coherent submodule for which the inclusion induces a local nil-isomorphism.
\end{enumerate}
\end{Bsp}

\section{Testmodules and Pushforwards}
\label{SectionTestmodulesPushforwards}

In this section we show that for $f: X \to Y$ quasi-finite and of finite type  and $M$ a $\mathcal{C}_X$-module one has a natural transformation $\tau(f_\ast M) \to f_\ast\tau(M)$ under the restriction on the Cartier algebra $\mathcal{C}_X$ which ensures the nil-coherence of $f_*M$ as explained in the preceding section.

The idea is as follows. By Zariski's main theorem we may factor $f$ as $i \circ g$, where $g$ is finite and $i$ is an open immersion. If $g$ is not dominant then we may factor $g$ as a finite dominant morphism followed by a closed immersion. In the case of a closed immersion this natural transformation is given by the identity which follows from Proposition \ref{TauFormalProperties} (d). We are thus left to deal with the case of a finite dominant morphism and that of an open immersion. In order to tackle the case of an open immersion we need to assume that $\mathcal{C}_Y$ is of the form $\mathcal{D}^{\mathfrak{a^t}}$ for $\mathcal{D}$ finitely generated, $\mathfrak{a}$ an ideal sheaf and $t \geq 0$ a rational number. Moreover, we will require that the base is $F$-finite in this case.

We begin with the finite case.

\begin{Prop}
\label{FiniteShriekPushNaturalIso}
Let $f: X \to Y$ and $i: \Spec Z \to Y$ be finite morphisms. The natural base change morphism obtained form the pullback diagram
\[
\xymatrix{ X_Z \ar[r]^{f'}\ar[d]^{i'} & Z \ar[d]^i \\ X \ar[r]^f & Y}
\]
$f'_\ast i'^\flat \to i^\flat f_\ast $ is an isomorphism of functors of Cartier modules (resp. Cartier crystals).
\end{Prop}
\begin{proof}
The base change map is obtained by applying $f_*$ to the unit of adjunction $ i'_*i^\flat \to \id$ from Proposition \autoref{adjunctionsoffiniteandetale} to obtain
\[
    i_*f_*i'^\flat \cong f'_*i'_*i'^\flat \to f_*.
\]
Using adjunction for $(i_*,i^\flat)$, this yields the base change morphism $f'_\ast i'^\flat \to i^\flat f_\ast $. By construction this morphism is Cartier linear. To check that this map is an isomorphism is a local property, so that we may assume that  $Y = \Spec R$ and since both $f$ and $i$ are affine also $X = \Spec S$ and $Z = \Spec A$. With this notation the assertion is that the natural map $f'_\ast \Hom_S(A \otimes_R S, M) \to \Hom_R(A, f_\ast M)$ is an isomorphism, but this is just the usual Tensor-Hom adjointness  $\Hom_S(A \otimes_R S, M) \to \Hom_R(A, M), \varphi \mapsto [a \mapsto \varphi(a \otimes 1)]$.
\end{proof}

\begin{Prop}
\label{EtaleShriekFinitePushNaturalIso}
Let $f: X \to Y$ be a morphism of finite type and $j: Z \to Y$ essentially \'etale. Then there is a natural isomorphism of functors of Cartier modules or crystals $j^! f_\ast \cong f'_\ast j'^!$, where $f'$ and $j'$ are the pullbacks of $f$ and $j$.
\end{Prop}
\begin{proof}
The \v{C}ech complex is a resolution in the category of Cartier modules (\cite[Theorem 3.2.2]{blickleboecklecartiercrystals} and the discussion preceding it). Hence, the corresponding isomorphism of modules obtained via \cite[Proposition III.9.3]{hartshornealgebraic} is compatible with Cartier structures.
\end{proof}

\begin{Le}
\label{PushforwardAssociatedPrimes}
Let $f: \Spec S \to \Spec R$ be a finite morphism and $M$ a $\mathcal{C}_S$-module. Then $\Ass f_\ast M = f(\Ass M)$.
\end{Le}
\begin{proof}
The statement holds for the underlying modules by \cite[Proposition 9.A]{matsumura}. Using Propositions \ref{FiniteShriekPushNaturalIso}, \ref{EtaleShriekFinitePushNaturalIso} and \ref{PushforwardPreservesNilpotence} we conclude that \[\underline{H^0_\eta(f_\ast M)_\eta} = f'_\ast \underline{H^0_{\eta S}(M)_\eta}.\] Since $\underline{H^0_{\eta S}(M)_\eta}$ lives on the reduced fiber of $\eta$ and is given by $\bigoplus_{\nu \in f^{-1}(\eta)} \underline{H^0_\nu(M)}_\nu$ we may use Proposition \ref{PushforwardPreservesNilpotence} and \cite[proof of Lemma 6.13]{staeblertestmodulnvilftrierung} to conclude that $H^0_\eta(f_\ast M)$ is nilpotent if and only if each $H^0_\nu(M)_\nu$ is nilpotent.
\end{proof}

\begin{Prop}
\label{FinitePushforwardPreservesFRegular}
Let $f: \Spec S \to \Spec R$ be a finite dominant map and $M$ an $S$-Cartier module. Then $M$ is $F$-regular if and only if  $f_\ast M$ is $F$-regular.
\end{Prop}
\begin{proof}
Let us assume first that $f_*M$ is $F$-regular. By Proposition \ref{FPurePushforward} $M$ is certainly $F$-pure. Assume that $M$ is not $F$-regular. That is, there exists a proper submodule $N \subseteq M$ such that the inclusion $H^0_\nu(N)_\nu \subseteq H^0_\nu(M)_\nu$ is a nil-isomorphism for all associated primes $\nu$ of $M$. Note that the inclusion $f_\ast N \subseteq f_\ast M$ is still proper. Since the $\nu \in f^{-1}(\eta)$ form an open covering we conclude that $H^0_{\eta S}(N)_\eta \subseteq H^0_{\eta S}(M)_\eta$ is a nil-isomorphism.

We have to show that the inclusion $H^0_\eta(f_\ast N)_\eta \subseteq H^0_\eta(f_\ast M)_\eta$ is a nil-isomorphism for all $\eta \in \Ass f_\ast M$. Consider the following diagram
\[\begin{xy}
\xymatrix{H^0_\eta(f_\ast N)_\eta \ar[r]^{\subseteq} & H^0_\eta(f_\ast M)_\eta \\ \underline{H^0_\eta(f_\ast N)}_\eta \ar[u]^{\subseteq} & \underline{H^0_\eta(f_\ast M)}_\eta \ar[u]^{\subseteq}\\ \underline{f_\ast H^0_{\eta S}(N)}_\eta \ar[r]^{\subseteq} \ar[u]^{=}& \underline{f_\ast H^0_{\eta S}(M)}_\eta \ar[u]^{=}}
\end{xy}\]
Note that the bottom vertical arrows are nil-isomorphisms obtained from Proposition \ref{FiniteShriekPushNaturalIso} and \ref{EtaleShriekFinitePushNaturalIso}. By assumption, Proposition \ref{PushforwardPreservesNilpotence}, Lemma \ref{PushforwardAssociatedPrimes} and the above observation the bottom horizontal arrow is also a nil-isomorphism whence the claim.

For the converse, assume now that $M$ is $F$-regular and let $N \subseteq f_\ast M$ be a submodule such that $H^0_\eta(N)_\eta \subseteq H^0_\eta(f_\ast M)_\eta$ is a nil-isomorphism for all $\eta \in \Ass f_\ast M$. Note that $f_\ast M$ is $F$-pure by Proposition \ref{FPurePushforward}. Given $\eta$ we write $\alpha: \Spec (R/\eta)_\eta \to \Spec R$ for the inclusion. Note that $f'^! \alpha^! = \alpha'^! f^!$, where $\alpha'$ and $f'$ are the pullbacks of $\alpha$ and $f$. If we denote by $\gamma$ the inclusion $\Spec (S/\nu)_\nu \to \Spec (S/ \eta S)_\eta$, where $\nu$ is a point in the fiber $f^{-1}(\eta)$, then the inclusion $N \subseteq f_\ast M$ induces a nil-isomorphism $\gamma^! \alpha'^! f^! N \subseteq \gamma^! \alpha'^! f^! f_\ast M$ which means that $\underline{H^0_\nu(f^! N)}_\nu \subseteq \underline{H^0_\nu(f^! f_\ast M)}_\nu$ is a nil-isomorphism. We conclude that the inclusion \begin{equation}\label{EQ2} H^0_\nu(f^! N)_\nu \subseteq H^0_\nu(f^! f_\ast M)_\nu\end{equation} is also a nil-isomorphism.

Consider now the counit $M \to f^! f_\ast M, m \mapsto [1 \mapsto m]$ which is an embedding. Since both $H^0_\nu$ and localization commute with intersections we obtain that the inclusion $f^! N \cap M \subseteq M$ induces a nil-isomorphism \[H^0_\nu(f^! N \cap M)_\nu \subseteq H^0_\nu(M)_\nu\] by intersecting the inclusion (\ref{EQ2}) in the previous paragraph with $H^0_\nu(M)_\nu$. Since $\Ass M \subseteq \Ass f^! f_\ast M = f^{-1} \Ass f_\ast M$ by Lemmata \ref{AssForCrystals} and \ref{FiniteShriekAssPrimes} we have a nil-isomorphism for all associated primes of $M$. We conclude that $f^! N \cap M = M$ by $F$-regularity of $M$. This implies $M \subseteq f^!N$. Taking pushforwards and applying the trace map shows that $N = f_\ast M$.
\end{proof}

\begin{Prop}
\label{TauFiniteDomPushforward}
Let $f: \Spec S \to \Spec R$ be a finite dominant map and $M$ an $S$-Cartier module. Then $f_\ast \tau(M) = \tau(f_\ast M)$
\end{Prop}
\begin{proof}
We may assume that $M$ is $F$-pure. By Propositions \ref{FiniteShriekPushNaturalIso} and \ref{EtaleShriekFinitePushNaturalIso} we have $f_\ast H^0_{\eta S}(\tau(M))_{\eta} = H^0_\eta(f_\ast \tau(M))_\eta$ and similarly $f_\ast H^0_{\eta S}(M)_\eta = H^0_\eta(f_\ast M)_\eta$. Since $\nu \in f^{-1}(\eta)$ form an open covering, we have a nil-isomorphism $H^0_{\eta S}(\tau(M))_{\eta} \subseteq H^0_{\eta S}(M)_\eta$. We thus obtain a nil-isomorphism $H^0_\eta(f_\ast \tau(M))_\eta \subseteq H^0_\eta(f_\ast M)_\eta$. Since $\tau(f_\ast M)$ is minimal with this property we obtain the inclusion $\tau(f_\ast M) \subseteq f_\ast \tau(M)$.

For the other inclusion we note that by Proposition \ref{FinitePushforwardPreservesFRegular} above the inclusion $f_\ast \tau(M) \supseteq \tau(f_\ast M)$ has to be an equality.
\end{proof}

\begin{Prop}
\label{FiniteMorphismTauNatTransf}
Let $f: X \to Y$ be a finite morphism of noetherian schemes and $M$ a Cartier module on $X$. Then the restriction of the identity on $f_\ast M$ induces a natural isomorphism of Cartier modules $f_\ast \tau(M) \to \tau(f_\ast M)$.
\end{Prop}
\begin{proof}
The assertion is clearly local on $Y$ and hence also on $X$. We are thus reduced to a situation $f: \Spec S \to \Spec R$ with $f$ finite.

Starting with the identity $f_\ast M \to f_\ast M$ we obtain via adjunction and by applying $\tau$ a morphism $\tau(M) \to \tau(f^! f_\ast M)$. By Propositions \ref{TauFormalProperties} (c), \ref{FiniteShriekTauInclusion} we have an inclusion $\tau(f^! f_\ast M) \subseteq f^! \tau(f_\ast M)$. Applying adjunction once more to the morphism $\tau(M) \to f^! \tau(f_\ast M)$ yields the desired natural transformation. Since $\tau$ applied to a morphism is simply restriction we immediately obtain from Propositions \ref{TauFiniteDomPushforward}, \ref{TauFormalProperties} (d) that this natural transformation is the identity.
\end{proof}

Before we proceed with the case of an open immersion we remind the reader of Definition \ref{DefTestmodulePushforward} which explained how to attach a test module to finite type pushforwards. In particular, from now on we assume that the Cartier algebra on the base is of the form $\mathcal{C}_Y^{\mathfrak{a}^t}$ with $\mathcal{C}_Y$ finitely generated.

\begin{Bem}
\label{CircumventGaugeBoundedness}
It seems plausible that one can circumvent the assumption on $\mathcal{C}_Y$ as follows. Assume that $Y$ is affine. Let $M$ be an $F$-pure coherent $\mathcal{C}_X$-module. Fix a finitely generated sub-algebra $\mathcal{C}' \subseteq \mathcal{C}_X$ that satisfies the assumptions of Lemma \ref{TauExistenceFinitelyGeneratedC}. Fix an open affine covering $U \subseteq f^{-1}(Y)$ and finitely many algebra generators $\kappa_e$ of $\mathcal{C}'\vert_U$. Then each $\kappa_e$ is of the form $\sum_i \kappa_{ei} \otimes s_{ei}$. Denote the sub-algebra of $\mathcal{C}_Y$ generated by the $\kappa_{ei}$ by $\mathcal{C}''$. Then clearly $f^\ast \mathcal{C}''$ satisfies the assumptions of Lemma \ref{TauExistenceFinitelyGeneratedC} so that $\mathcal{C}_X \tau(M, f^\ast\mathcal{C}'') = \tau(M, \mathcal{C}_X)$. Since $\mathcal{C}''$ is finitely generated Theorem \ref{PushforwardCoherentCrystals} applies. Hence, we find a coherent $\mathcal{C}''$-submodule $N$ of $f_\ast M$ which is locally nil-isomorphic to $f_\ast M$ (as $\mathcal{C}''$-modules). Then we define $\tau(f_\ast M, \mathcal{C}_Y) = \mathcal{C}_Y \tau(N, \mathcal{C}'')$. By the discussion following Definition \ref{DefTestmodulePushforward}, this does not depend on the choice of $N$ for fixed $\mathcal{C}''$. Lemma \ref{TauExistenceFinitelyGeneratedC} makes it plausible that it also should not depend on $\mathcal{C}''$. If this is the case then we obtain the same result for general $Y$ by gluing.

However, in applications (i.e.\ when studying singularities) one usually studies Cartier algebras constructed from $\bigoplus_e \mathcal{H}om(F_\ast^e M, M)$ and some ideal $\mathfrak{a}$. At present we are not aware of an example in this situation where Theorem \ref{PushforwardCoherentCrystals} does not apply directly. There is an example (cf.\ \cite{katzmannonfgcartieralgebra}) where the algebra $\bigoplus_e \mathcal{H}om(F_\ast^e M, M)$ is not finitely generated. However, in this case one easily checks that it is gauge bounded.
\end{Bem}

\begin{Prop}
\label{OpenImmersionTauNatTransf}
Let $j: X \to Y$ be an open immersion of $F$-finite schemes and assume that the Cartier algebra on $Y$ is of the form $\mathcal{C}_Y^{\mathfrak{a}^{t}}$ with $\mathcal{C}_Y$ finitely generated. If $M$ is an $\mathcal{O}_X$-Cartier module then we have a natural transformation $\tau j_\ast \to j_\ast \tau$ of functors from Cartier modules on $X$ to Cartier modules on $Y$ which is given by the natural inclusion.
\end{Prop}
\begin{proof}
Note that $j^! = j^\ast$. We start with the identity $j_\ast M \to j_\ast M$. Adjunction yields $j^\ast j_\ast M \to M$. Applying $\tau$ we get an isomorphism $\tau(j^\ast j_\ast M) \to \tau(M)$. If $N$ is any coherent submodule for which the inclusion $N \subseteq j_\ast M$ is a local nil-isomorphism then $j^\ast N \subseteq j^\ast j_\ast M = M$ is again a nil-isomorphism by Proposition \ref{CartierOperationEtalePullbackCommutes} (the proof also works in the quasi-coherent case). %lnil is local (in fact, upstairs it's the same as nil), colim and \otimes commutes so that the filtration is also exhaustive.
Using Lemma \ref{TauNilisoLemma} and Theorem \ref{EtalePullbackFRegularity} we get $\tau(j^\ast j_\ast M) = \tau(j^\ast N) = j^\ast \tau(N) \to \tau(M)$. Adjunction now yields $\tau(j_\ast M) = \tau(N) \to j_\ast \tau(M)$. Note that $\tau(j^\ast j_\ast M) = j^\ast \tau(j_\ast M)$ is an equality as submodules of $j^\ast j_\ast M$. Since the natural transformation is induced by the identity $j_\ast \to j_\ast$ it is simply the natural inclusion.
\end{proof}

\begin{Bsp}
The natural transformation $\tau j_\ast \to j_\ast \tau$ does not in general induce an isomorphism of Cartier modules or crystals. Consider $M = R = k[x,x^{-1}]$ with Cartier structure obtained from identification with $\omega_R$ and let $j: D(x) \to \mathbb{A}^1_k$ be the natural inclusion. Then $\tau(M) = k[x,x^{-1}]$ while $\tau(j_\ast M) = \tau(k[x] \cdot x^{-1}) = k[x]$. One checks that the natural inclusion is not a local nil-isomorphism.

In fact, $k[x]$ is simple as a Cartier crystal while $k[x] \cdot x^{-1}$ admits the proper sub-crystal $k[x]$. Hence, they are not even abstractly isomorphic.
\end{Bsp}

\begin{Theo}
\label{TauQuasifinitepushforward}
For $f: X \to Y$ a quasi-finite finite type morphism of $F$-finite schemes with Cartier algebra on $Y$ of the form $\mathcal{C}_Y^{\mathfrak{a}^t}$ with $\mathcal{C}_Y$ finitely generated one has a natural transformation $\tau f_\ast \to f_\ast \tau$ of Cartier modules or crystals.
\end{Theo}
\begin{proof}
By Zariski's main theorem we may factor $f = g \circ j$ with $g$ finite and $j$ an open immersion. Then we define the natural transformation by composing the natural transformations $f_\ast \tau \to \tau f_\ast $ and $\tau j_\ast \to j_\ast \tau$ of Propositions \ref{FiniteMorphismTauNatTransf} and \ref{OpenImmersionTauNatTransf}. Since these natural transformations are all given by the natural inclusions this is independent of the choice of factorization.
\end{proof}

Recall that we use the notation $Gr(M, \mathfrak{a}^t)$ for the quotient $\tau(M, \mathfrak{a}^{t-\eps})/\tau(M, \mathfrak{a}^t)$.

\begin{Ko}
\label{GrQuasifinitepushforward}
For $f: X \to Y$ a quasi-finite finite type morphism of $F$-finite schemes with Cartier algebra on $Y$ of the form $\mathcal{C}_Y^{\mathfrak{a}^t}$ with $\mathcal{C}_Y$ finitely generated one has a natural transformation $Gr(f_\ast -, \mathfrak{a}^t) \to f_\ast Gr(-, (\mathfrak{a} \cdot \mathcal{O}_X)^t)$ of Cartier modules or crystals.
\end{Ko}
\begin{proof}
We may factor $f$ as $g \circ j$ with $g$ finite and $j$ an open immersion. Since $g_\ast$ is exact the assertion immediately follows from Proposition \ref{FiniteMorphismTauNatTransf}. For an open immersion the argument is similar to the one of Corollary \ref{FiniteShriekGrInclusion} using Proposition \ref{OpenImmersionTauNatTransf}.
\end{proof}

\begin{Bem}
We point out that the restriction of the identity does not induce a natural transformation $\tau f_\ast \to f_\ast \tau$ in the category of Cartier crystals for general finite type morphisms. Indeed, consider the structural map $\mathbb{A}^1_k \to \Spec k$, where $k$ is any $F$-finite field and the Cartier module $\omega = \omega_{k[x]}$ endowed with Cartier structure $\kappa x^{p-1}$. Then $\tau(\omega) = x \omega$ while we have $f_\ast \omega = k dx$ and $f_\ast \tau( \omega) = 0$. Since $f_\ast \omega$ is defined over a field one has (as crystals) $\tau( f_\ast \omega) = f_\ast \omega$.
\end{Bem}

\bibliography{bibliothek}

\providecommand{\bysame}{\leavevmode\hbox to3em{\hrulefill}\thinspace}
\providecommand{\MR}{\relax\ifhmode\unskip\space\fi MR }
\providecommand{\MRhref}[2]{%
  \href{http://www.ams.org/mathscinet-getitem?mr=#1}{#2}
}
\providecommand{\href}[2]{#2}
\begin{thebibliography}{BSTZ10}

\bibitem[BB11]{blickleboecklecartierfiniteness}
M.~Blickle and G.~B\"ockle, \emph{Cartier modules: {F}initeness results}, J.
  reine angew. Math. \textbf{661} (2011), 85--123.

\bibitem[BB13]{blickleboecklecartiercrystals}
\bysame, \emph{Cartier crystals}, arXiv:1309.1035v1 (2013).

\bibitem[Bli13]{blicklep-etestideale}
M.~Blickle, \emph{Test ideals via $p^{-e}$-linear maps}, J. Algebraic Geom.
  \textbf{22} (2013), no.~1, 49--83.

\bibitem[BMS09]{blicklemustatasmithdiscretenesshypersurfaces}
M.~Blickle, M.~Musta\c{t}\u{a}, and K.~Smith, \emph{${F}$-thresholds of
  hypersurfaces}, Trans. Amer. Math. Soc. \textbf{361} (2009), no.~12,
  6549--6565.

\bibitem[BP09]{boecklepinktaucrystals}
G.~B\"ockle and R.~Pink, \emph{Cohomological {T}heory of {C}rystals over
  function fields}, EMS Tracts in Mathematics, no.~9, European Mathematical
  Society, 2009.

\bibitem[BS98]{brodmannsharp}
M.~P. Brodmann and R.~Y. Sharp, \emph{Local cohomology}, Cambridge University
  Press, 1998.

\bibitem[BS16]{blicklestaeblerbernsteinsatocartier}
M.~Blickle and A.~St{\"a}bler, \emph{Bernstein-{S}ato polynomials and test
  modules in positive characteristic}, Nagoya Math. J. \textbf{222} (2016),
  no.~1, 74--99.

\bibitem[BSTZ10]{blickleschwedetakagizhangDiscretenessfjumpingsingular}
Manuel Blickle, Karl Schwede, Shunsuke Takagi, and Wenliang Zhang,
  \emph{Discreteness and rationality of {$F$}-jumping numbers on singular
  varieties}, Math. Ann. \textbf{347} (2010), no.~4, 917--949.

\bibitem[DK73]{SGA7II}
P.~Deligne and N.~Katz (eds.), \emph{Groupes de monodromie en g\'eom\'etrie
  alg\'ebrique {\rm(}={SGA7-II}{\rm)}}, Lecture notes in Mathematics, vol. 340,
  Springer, 1973.

\bibitem[Eis95]{eisenbud}
D.~Eisenbud, \emph{Commutative algebra with a view toward algebraic geometry},
  Springer-Verlag, 1995.

\bibitem[EK04a]{emertonkisinintrorhunitfcrys}
M.~Emerton and M.~Kisin, \emph{An introduction to the {R}iemann-{H}ilbert
  correspondence for unit ${F}$-crystals}, Geometric aspects of {D}work theory,
  vol.~2, Walter de Gruyter, 2004, pp.~677--700.

\bibitem[EK04b]{emertonkisinrhunitfcrys}
\bysame, \emph{The {R}iemann-{H}ilbert correspondence for unit ${F}$-crystals},
  Asterisque \textbf{293} (2004).

\bibitem[Gab04]{gabbertstructures}
O.~Gabber, \emph{Notes on some $t$-structures}, Geometric aspects of Dwork
  theory, Walter de Gruyter, 2004, pp.~711--734.

\bibitem[Har66]{hartshorneresidues}
R.~Hartshorne, \emph{Residues and duality}, Lecture Notes in Mathematics,
  vol.~20, Springer, 1966.

\bibitem[Har77]{hartshornealgebraic}
\bysame, \emph{Algebraic {G}eometry}, Springer, New York, 1977.

\bibitem[Har01]{haratestmultiplier}
N.~Hara, \emph{Geometric interpretation of tight closure and test ideals},
  Trans. Amer. Math. Soc. \textbf{353} (2001), no.~5, 1885--1906.

\bibitem[HH90]{hochsterhunekebriancon}
M.~Hochster and C.~Huneke, \emph{Tight closure, {I}nvariant theory, and the
  {B}rian\c{c}on-{S}koda {T}heorem}, J. Amer. Math. Soc. \textbf{3} (1990),
  31--116.

\bibitem[HT04]{haratakagigeneralization}
N.~Hara and S.~Takagi, \emph{Some remarks on a generalization of test ideals},
  Nagoya Math. J. \textbf{175} (2004), 59--74.

\bibitem[Kat10]{katzmannonfgcartieralgebra}
Mordechai Katzman, \emph{A non-finitely generated algebra of {F}robenius maps},
  Proc. Amer. Math. Soc. \textbf{138} (2010), no.~7, 2381--2383.

\bibitem[Len08]{lenstragaloistheory}
H.~W. Lenstra, \emph{Galois theory for schemes}, 2008,
  \url{http://websites.math.leidenuniv.nl/algebra/GSchemes.pdf} (accessed:
  2016-02-29).

\bibitem[Mat70]{matsumura}
H.~Matsumura, \emph{Commutative algebra}, New York, 1970.

\bibitem[Mat89]{matsumuracommutativeringtheory}
\bysame, \emph{Commutative ring theory}, Cambridge Studies in Advanced
  Mathematics, vol.~8, Cambridge University Press, Cambridge, 1989.

\bibitem[Sch11]{schwedenonqgorenstein}
K.~Schwede, \emph{Test ideals in non-$\mathbb{Q}$-{G}orenstein rings}, Trans.
  Amer. Math. Soc. \textbf{363} (2011), no.~11, 5925--5941.

\bibitem[Smi97]{smithrational}
K.~E. Smith, \emph{F-rational rings have rational singularities}, Amer. J.
  Math. \textbf{119} (1997), 159--180.

\bibitem[Smi00]{smithtestmultiplier}
\bysame, \emph{The multiplier ideal is a universal test ideal}, Commun. in
  Algebra \textbf{28} (2000), no.~12, 5915--5929.

\bibitem[Sta14]{stadnikvfiltrationfcrystal}
Th. Stadnik, \emph{The ${V}$-filtration for tame unit ${F}$-crystals}, Sel.
  Math. New Ser. \textbf{20} (2014), no.~3, 855--883.

\bibitem[St{\"a}15a]{staeblerunitftestmoduln}
A.~St{\"a}bler, \emph{Test module filtrations for unit {$F$}-modules}, preprint
  (2015).

\bibitem[St{\"a}15b]{staeblertestmodulnvilftrierung}
\bysame, \emph{{$V$}-filtrations in positive characteristic and test modules},
  to appear Trans. Amer. Math. Soc. (2015).

\end{thebibliography}
\bibliographystyle{amsalpha}
\end{document}